\documentclass{article}
\usepackage{graphicx} 
\usepackage{amssymb,amsmath,amsthm,dsfont,enumitem}
\usepackage[english]{babel}
\usepackage[nottoc]{tocbibind}
\usepackage{nicematrix}
\usepackage{biblatex}

\theoremstyle{definition}
\newtheorem{dfn}{Definition}[section]
\theoremstyle{remark}
\newtheorem{obs}[dfn]{Observation}
\newtheorem{rmk}[dfn]{Remark}
\theoremstyle{plain}
\newtheorem{prp}[dfn]{Proposition}
\newtheorem{lem}[dfn]{Lemma}
\newtheorem{thm}[dfn]{Theorem}
\newtheorem{cor}[dfn]{Corollary}

\title{Irreducible representations of tree automorphism groups into Pontryagin spaces}
\author{Federico Viola}
\date{October 30, 2026}

\begin{document}

\maketitle

\section*{Abstract}

\textit{Let $G=\mathrm{Aut}(T)$ be the automorphism group of a regular tree $T$. We study continuous irreducible representations of $G$ that preserve a continuous strongly nondegenerate sesquilinear form of finite index on a Hilbert space. These are already classified for index $0$ (unitary case) and for index $1$. We show that there are no more representations for index $>1$, which completes the classification.} \\

\section{Introduction}

Let $G=\mathrm{Aut}(T)$ be the automorphism group of a regular locally finite tree $T$, i.~e.~a tree where all vertices have the same finite degree $d\geq 3$. This group is equipped with the topology of pointwise convergence, making it a totally disconnected locally compact group. \\

This work is about \textit{representations} of $G$. To clarify our terminology, we define:

\begin{itemize}
    \item An \textit{algebraic representation} $(\pi, H)$ of a group $G$ in a (real, complex, or quaternionic) vector space $H$ is a group homomorphism $\pi: G\rightarrow\mathrm{Aut}(H)$, where $\mathrm{Aut}(H)$ is the group of linear automorphisms of $H$.
    \item An \textit{irreducible algebraic representation}, or \textit{algebraically irreducible representation}, is an algebraic representation in which the only $G$-invariant vector subspaces of $H$ are $\{0\}$ and $H$.
    \item If $G$ is a topological group and $H$ is a Hilbert space, a \textit{continuous representation} $(\pi, H)$ of $G$ in $H$ is a group homomorphism $\pi: G\rightarrow L(H)$, where $L(H)$ denotes the space of bounded linear operators on $H$, such that the action map $G\times H\rightarrow H$ given by $(g,v)\mapsto\pi(g)v$ is continuous. If $G$ is locally compact, this is equivalent to saying that the orbit map $g\mapsto\pi(g)v$ is continuous for every fixed $v$, or that the group homomorphism $\pi: G\rightarrow L(H)$ is continuous when $L(H)$ is equipped with the strong operator topology.
    \item A continuous representation is \textit{topologically irreducible}, or simply \textit{irreducible}, if the only \textit{closed} $G$-invariant vector subspaces of $H$ are $\{0\}$ and $H$.
    \item A continuous representation is \textit{unitary} if for every $g\in G$ the operator $\pi(g)$ is unitary, i.~e.~it preserves the scalar product on the Hilbert space $H$. \\
\end{itemize}

Irreducible algebraic representations of $G$ have been first studied and classified by Ol'shanskii \cite{olsh}. Later, a complete and detailed classification of irreducible unitary representations of $G$ was provided by Figa-Talamanca and Nebbia \cite{talanebbia}. \\

Beyond the unitary world, a focus of much recent interest is representations preserving a sesquilinear form of finite index, or representations into Pontryagin spaces, first introduced in \cite{pontryagin}. The unitary case corresponds to index $0$. All index $1$ real representations have been classified by Burger, Iozzi, and Monod \cite{bim}. \\

The main result of this paper is that G does not admit continuous irreducible representations of index $>1$. This is in strong contrast to the case of $\mathrm{SL}_2(\mathbb{R})$ and other rank one Lie groups, which admit representations of arbitrarily high index \cite{delzantpy,monodpy}. \\

\begin{thm} \label{first}
    Let $G=\mathrm{Aut}(T)$, where $T$ is a locally finite regular tree. Let $\pi$ be a continuous irreducible representation of $G$ into a (real, complex, or quaternionic) Hilbert space $H$. If $\pi$ preserves a continuous strongly nondegenerate sesquilinear form of finite index $p\geq 0$, then $p=0$ or $p=1$.
\end{thm}

Strongly nondegenerate sesquilinear forms were defined in \cite{bim} as being associated with positive definite forms that generate a complete Hilbert norm. We recall the main concepts and definitions below.

\begin{itemize}
    \item A \textit{sesquilinear form} on a vector space $H$ over the field $\mathbb{K}\in\{\mathbb{R},\mathbb{C},\mathbb{H}\}$ (with $\mathbb{H}$ denoting the quaternions) is a map $B: H\times H\rightarrow\mathbb{K}$ that is antilinear in the first variable and linear in the second. We will always assume sesquilinear forms to be \textit{Hermitian}, meaning that $B(v,w)=\overline{B(w,v)}$ for every $v,w\in H$, and implying that $B(v,v)\in\mathbb{R}$ for every $v\in H$. We will also assume $B$ to be continuous if $H$ is a Hilbert space.
    \item The sesquilinear form $B$ is said to be \textit{nondegenerate} if there is no nonzero $v\in H$ such that $B(v,w)=0$ for every $w\in H$.
    \item A subspace $W\subseteq H$ is said to be \textit{positive} (resp. \textit{negative}) \textit{definite} if $B(w,w)>0$ (resp. $B(w,w)<0$) for every nonzero $w\in W$.
    \item The \textit{index} of a sesquilinear form $B$ is defined as the minimum between the maximal dimension of a positive definite subspace in $H$ and the maximal dimension of a negative definite subspace in $H$. It might be either finite or infinite.
    \item Let $B$ be a (continuous, Hermitian) nondegenerate sesquilinear form of finite index $p$ on $H$. The assumptions imply that $H$ can be decomposed as $H_+\oplus H_-$ where $B$ is positive definite on $H_+$ and negative definite on $H_-$, the two addends are orthogonal with respect to $B$, and $\min\{\dim(H_+),\dim(H_-)\}=p$. If we denote by $B_{\pm}$ the restriction of $B$ to $H_{\pm}$, the form $B_+-B_-$ is positive definite on $H$. The form $B$ is said to be \textit{strongly nondegenerate} if the form $B_+-B_-$ obtained in this way generates a complete norm. This does not depend on the choice of the decomposition, as shown in (\cite{bim}, Lemma 2.4). We refer to (\cite{bim}, Section 2) for further details on strongly nondegenerate forms. \\
\end{itemize}

Traditionally, Hilbert spaces equipped with a strongly nondegenerate sesquilinear form are called Krein spaces, and Pontryagin spaces if the form has finite index. A Pontryagin space of index $0$ is a Hilbert space with a positive definite sesquilinear form that generates a complete norm, so it is simply a Hilbert space for this form, and continuous representations that preserve the form are just unitary representations. \\

Together with \cite{olsh,talanebbia,bim}, Theorem \ref{first} completes the classification of all continuous irreducible representations of $G=\mathrm{Aut}(T)$ to all Pontryagin spaces, where by \textit{representation to a Pontryagin space} we mean a representation that preserves the sesquilinear form defining the space. \\

The study of Hilbert spaces with indefinite sesquilinear forms dates back to Pontryagin \cite{pontryagin}, Iohvidov-Krein \cite{iohkrein}, and Naimark \cite{naimark1,naimark2,naimark3,naimark4}; further treatment can be found in \cite{aziioh,bognar,iohkrlan}. More recently, Gromov (\cite{gromov}, Section 6) has suggested the study of the infinite-dimensional symmetric spaces of finite rank that can be defined subsequently. The case of the infinite-dimensional real hyperbolic space has been studied by Burger, Iozzi, and Monod in \cite{bim}, and later by Monod and Py in \cite{monodpy}, while the general case has been studied by Duchesne in recent times \cite{ducphd,duc2013,duc2015,duc2023}. It turns out (\cite{duc2015}, Corollary 1.10) that the only irreducible infinite-dimensional symmetric spaces of non-positive curvature and finite rank are, up to homothety, $\mathrm{O}(p,\infty)/\mathrm{O}(p)\times\mathrm{O}(\infty)$, $\mathrm{U}(p,\infty)/\mathrm{U}(p)\times\mathrm{U}(\infty)$, and $\mathrm{Sp}(p,\infty)/\mathrm{Sp}(p)\times\mathrm{Sp}(\infty)$, and that their isometry groups are, up to a subgroup of index two in the complex case, exactly the projective versions of $\mathrm{O}(p,\infty)$, $\mathrm{U}(p,\infty)$, and $\mathrm{Sp}(p,\infty)$ (\cite{duc2023}, Theorem 3.3). Therefore, studying representations of $G$ that preserve a nondegenerate sesquilinear form of finite index is the same as studying actions of $G$ by isometries on infinite-dimensional symmetric spaces of nonpositive curvature and finite rank. \\

In another paper \cite{higherrank}, we consider the case where $G$ is a higher-rank algebraic group over a non-archimedean local field $\mathbb{F}$, and we show that, under mild assumptions (namely, $\mathbb{F}$ has characteristic zero and its residue field has at least three elements, or $G$ contains $\mathrm{SL}_3(\mathbb{F})$), there are no continuous fixed point-free actions of $G$ on any infinite-dimensional symmetric space of nonpositive curvature and finite rank, which implies that there are no continuous irreducible representations of $G$ into any Pontryagin space of index $\geq 1$. \\

Turning back to our paper, we will actually show Theorem \ref{first} in greater generality: not only for $G=\mathrm{Aut}(T)$, but for all closed noncompact subgroups of $\mathrm{Aut}(T)$ that act transitively on the boundary at infinity $\partial T$ and satisfy Tits' independence property (P) (\cite{tits}, 4.2). We may also drop the assumption that the tree is regular and start from any locally finite tree where every vertex has degree $\geq 3$: our assumptions on $G$ imply that the tree is either regular or biregular (the latter meaning that there exist $r,s\geq 3$ such that all vertices at even distance from a given vertex have degree $r$, and all those at odd distance have degree $s$). \\

The groups we consider include universal groups associated with $2$-transitive permutation groups, which were introduced by Burger and Mozes in (\cite{burgermozes}, 3.2) and then studied by Amann in his PhD thesis (\cite{amann}, Chapter 3). More recently, Reid and Smith \cite{reidsmith} provided a characterization of tree automorphism groups with Tits' independence property (P) in terms of combinatorial structures called local action diagrams. \\

Unitary representations of the groups we consider have been classified by Amann (\cite{amann}, Chapter 2). The representations fall into three families: super-cuspidal, special, and spherical, according to the conjugacy class of a maximal compact subgroup that fixes a nontrivial vector. \\

Our result, formulated in this generality, is the following:

\begin{thm} \label{Main}
Let $T$ be a locally finite tree where every vertex has degree $\geq 3$. Let $G$ be a closed noncompact subgroup of $\mathrm{Aut}(T)$, with Tits' Property (P) and transitive action on $\partial T$. Let $\mathbb{K}\in\{\mathbb{R},\mathbb{C},\mathbb{H}\}$, and let $(\pi,H)$ be a continuous irreducible representation of $G$ to a $\mathbb{K}$-Hilbert space. If $\pi$ preserves a continuous strongly nondegenerate sesquilinear form of finite index $p\geq 0$, then $p=0$ or $p=1$. Moreover, if $p=1$, then the representation $(\pi,H)$ is isomorphic to the tensor product with $\mathbb{K}$ of a real representation.
\end{thm}

Theorem \ref{Main} obviously implies Theorem \ref{first}, and the final statement about the case $p=1$ confirms that, also in our generality, the representations classified in \cite{bim} exhaust all possible non-unitary finite-index representations. This was probably clear to experts following \cite{bim}, but here we provide an explicit proof. \\

The paper is organized as follows. In Sect. 2 we give a few preliminaries on trees and on the independence property, and we explain why we can restrict to a representation that is irreducible in the algebraic sense (the relevant technical results are shown in the Appendix). We also recall a few useful notions about convolution algebras and Gel'fand pairs. In Sect. 3 we show that we can classify the representations as spherical, special, and super-cuspidal, and that for our purposes it suffices to consider spherical representations. In Sect. 4 we consider the case where $G$ acts transitively on the set of vertices of the tree and give an explicit description of spherical representations, getting a classification that depends on only one parameter. In Sect. 5 we study nondegenerate sesquilinear forms of finite index that can be preserved by these representations, completing the proof of our theorem in the case where the action of $G$ is transitive on the vertices. Finally, in Sect. 6 we generalize the results of Sects. 4 and 5 to the case where the action of $G$ is not transitive on the vertices. The Appendix, while functional to our scope of restricting to algebraically irreducible representations, contains results about admissibility of representations and irreducible actions of operators algebras that can also carry some independent interest. \\

\section{Preliminaries}

A tree $T$ is an undirected graph in which every two vertices are connected by exactly one path. We call the set of vertices $V$ and the set of (unoriented) edges $E$. The \textit{degree} of a vertex is the number of edges that have an endpoint at the vertex. A tree is \textit{locally finite} if all vertices have finite degree. A locally finite tree is \textit{regular} of degree $d$ if all vertices have degree $d$, and \textit{biregular} of degrees $(r,s)$ if all vertices at even distance from a base vertex have degree $r$, and all vertices at odd distance have degree $s$. \\

The boundary at infinity $\partial T$ of the tree $T$ is the set of infinite injective paths $(x_n)_{n\in\mathbb{N}}$ with $x_n\in V$ for $n\in\mathbb{N}$, where two paths $(x_n)_{n\in\mathbb{N}}$ and $(y_n)_{n\in\mathbb{N}}$ are identified if there exists $t\in\mathbb{Z}$ such that $x_n=y_{n+t}$ for all large enough $n$. \\

We consider a locally finite tree $T$ where every vertex has degree $\geq 3$, and its automorphism group $\mathrm{Aut}(T)$. This group, equipped with the topology of pointwise convergence, is totally disconnected and locally compact. It has a natural action on the tree and on its boundary at infinity. \\

Let $G<\mathrm{Aut}(T)$ be a subgroup. For each subtree $S\subseteq T$ (which might be a single vertex or an edge), let $G(S)$ be the subgroup of $G$ that fixes every point in $S$. \\

If $e\in E$ is an edge, it divides the tree into two subtrees whose intersection is just $e$ (together with its two endpoints). We call the two subtrees $T_e^{(1)}$ and $T_e^{(2)}$. It is immediate to see that every automorphism of $T$ that fixes every point in $T_e^{(1)}$ commutes with every automorphism of $T$ that fixes every point in $T_e^{(2)}$. \\

We can now introduce the independence property for a subgroup $G<\mathrm{Aut}(T)$. We use the definition given by Amann (\cite{amann}, Definition 9):

\begin{dfn}{(Independence property)} \label{indpro}
Let $G<\mathrm{Aut}(T)$ be a subgroup. We say that $G$ has the \textit{independence property} if, for every edge $e$, it holds $$G(e)=G(T_e^{(1)})G(T_e^{(2)}).$$
\end{dfn}

As observed by Amann (as a consequence of Lemma 10 in \cite{amann}), if $G$ is a closed subgroup of $\mathrm{Aut}(T)$, this property is equivalent to Tits' Property (P), which was defined in (\cite{tits}, 4.2). Therefore, from now on, we shall use Definition \ref{indpro} as an equivalent definition for Tits' Property (P). \\

In order to prove Theorem \ref{Main}, we are interested in studying representations of closed noncompact subgroups of $\mathrm{Aut}(T)$ that act transitively on the boundary $\partial T$ and have the independence property. From now on, let $G$ be such a subgroup. We recall a few properties of $G$ that follow from the assumptions; more details can be found in \cite{amann} (Proposition 4). \\

\begin{prp} \label{props}
The following hold:
\begin{enumerate}[label=(\alph*)]
    \item The identity in the group $G$ has a basis of open compact neighborhoods: the subgroups $(G(B_n(o)))_{n\in\mathbb{N}}$, where $o$ is a fixed vertex and $B_n(o)$ is the ball of radius $n$ around $o$.
    \item For every vertex $o$, the subgroup $G(o)$ acts transitively on $\partial T$.
    \item Either the action of $G$ on $V$ is transitive or it has two orbits. In the first case, the tree is regular; in the second case, the tree is regular or biregular, one orbit is the set of vertices at even distance from a fixed vertex $o$, and the other orbit is the set of vertices at odd distance from $o$. \\
\end{enumerate}
\end{prp}

As explained at the beginning of the Introduction, when studying irreducible representations of $G$ into a Hilbert space $H$, irreducibility is intended in the topological sense: there are no closed $G$-invariant subspaces of $H$ except $\{0\}$ and $H$. We shall now reduce to a situation where we have irreducibility in the algebraic sense, i.~e~no $G$-invariant subspaces at all, except $\{0\}$ and the full space. \\

To this end, we restrict our attention to a dense subspace of $H$, as in the following proposition. \\

\begin{prp} \label{smooth}
Let $(\pi,H)$ be a topologically irreducible continuous representation of $G$. Let $H^{(\infty)}\subseteq H$ be the subspace of smooth vectors in $H$, i.~e~vectors that are fixed by some open compact subgroup of $G$. Then $H^{(\infty)}$ is dense in $H$.
\end{prp}

\begin{proof}
Since the representation is topologically irreducible and the subspace $H^{(\infty)}$ is $G$-invariant (if $g\in G$ and $K<G$ satisfies $Kv=v$, then $gv$ is fixed by the subgroup $gKg^{-1}$), it is enough to show that there exists one nonzero smooth vector: once this is established, the closure of $H^{(\infty)}$ will be a nonzero closed $G$-invariant subspace of $H$ and therefore necessarily $H$. \\

Let $v\in H$, $v\neq 0$. By Hahn-Banach, there exists a continuous linear operator $L$ on $H$ such that $Lv\neq 0$. By the continuity of the representation, there is an open compact $K<G$ such that $\vert Lkv-Lv\vert<\frac{1}{2}\vert Lv\vert$ for all $k\in K$. We consider the vector $w:=\int_K kv\ dk$, where we integrate with respect to the Haar measure on $K$ (normalized so that the total measure of $K$ is $1$). This vector is fixed by $K$ and satisfies $\vert Lw-Lv\vert<\frac{1}{2}\vert Lv\vert$, which implies $\vert Lw\vert>\frac{1}{2}\vert Lv\vert$ and thus $w\neq 0$. \\
\end{proof}

We will prove in the Appendix that, under the hypotheses of Theorem \ref{Main} (namely the existence of a preserved sesquilinear form of finite index), the representation $(\pi,H)$ is \textit{admissible}, meaning that for every open compact $K<G$ the space $H^K$ of $K$-fixed vectors is finite-dimensional. This is not true in complete generality, as shown in an example by Nicolas Monod \cite{inadmissible}. \\

We will also show in the Appendix that admissibility implies that the restriction of the representation to $H^{(\infty)}$ is algebraically irreducible. If $(\pi,H)$ preserves a continuous strongly nondegenerate Hermitian sesquilinear form of finite index $p$ on $H$, the restriction of $\pi$ to $H^{(\infty)}$ preserves the restriction of the sesquilinear form to $H^{(\infty)}$, which is still nondegenerate and of the same index, since $H^{(\infty)}$ is dense in $H$. The algebraic representation $(\pi,H^{(\infty)})$ is also \textit{smooth}, meaning that every vector is fixed by some open compact subgroup of $G$. A good exposition of smooth representations of totally disconnected locally compact groups can be found in \cite{bbej}. \\

We now recall a few facts about convolution algebras and Gel'fand pairs that we will need in the following sections. \\

For a field $\mathbb{K}$, a \textit{$\mathbb{K}$-algebra} is a $\mathbb{K}$-vector space with a multiplication operator that is compatible with the vector space structure. \\

For a locally compact group $G$ and a field $\mathbb{K}$, we can consider the \textit{convolution algebra} $C_c(G,\mathbb{K})$ of compactly supported continuous $\mathbb{K}$-valued functions on $G$. The multiplication operation in this algebra is the convolution $*$, defined as $$(f_1*f_2)(g)=\int_G{f_1(h)f_2(h^{-1}g)dh},\ \ \ f_1,f_2\in C_c(G),\ g\in G.$$ Here the integral is over the Haar measure on G, and is finite because $f_1,f_2$ are compactly supported. \\

A representation of a $\mathbb{K}'$-algebra $\mathcal{A}$ in a $\mathbb{K}$-Hilbert space $H$, with $\mathbb{K}'\subseteq\mathbb{K}$, is a linear homomorphism of $\mathbb{K}$-vector spaces $\phi:\mathcal{A}\rightarrow L(H)$ that satisfies $\phi(f_1*f_2)=\phi(f_1)\phi(f_2)$ for $f_1,f_2\in\mathcal{A}$. As usual, $L(H)$ denotes the space of bounded linear operators on $H$. The representation is said to be \textit{irreducible} if the only closed $\mathcal{A}$-invariant subspaces of $H$ are $\{0\}$ and $H$. \\

In practice, we will only consider the following three cases for $\mathbb{K}'$ and $\mathbb{K}$:
\begin{itemize}
    \item $\mathbb{K}'=\mathbb{K}=\mathbb{R}$;
    \item $\mathbb{K}'=\mathbb{K}=\mathbb{C}$;
    \item $\mathbb{K}'=\mathbb{R}$, $\mathbb{K}=\mathbb{H}$. \\
\end{itemize}

If $\mathbb{K}'\subseteq\mathbb{K}$ and $(\pi,H)$ is a continuous representation of the group $G$ to a $\mathbb{K}$-Hilbert space $H$, we can extend it to a continuous representation of the convolution algebra $C_c(G,\mathbb{K}')$ by setting $\pi(f)v=\int_G f(g)\pi(g)v\ dg$ for $v\in H,\ f\in C_c(G,\mathbb{K}')$. It is immediate to check that it is indeed a representation, and a standard argument shows that it is irreducible if and only if the representation of $G$ is topologically irreducible. \\

\begin{lem} \label{symmetric}
Assume $G$ is unimodular. If the representation $(\pi,H)$ preserves a strongly nondegenerate sesquilinear form $\langle\cdot,\cdot\rangle$ on $H$, then the algebra $\pi(C_c(G,\mathbb{K}'))$ obtained by extending the representation is closed under taking adjoints for the sesquilinear form.
\end{lem}

\begin{proof}
If we denote with $\dag$ the adjoint operator, we can set $\hat{f}(g)=\overline{f(g^{-1})}$ for $f\in C_c(G,\mathbb{K}')$ and $g\in G$. We have $$\langle\pi(\hat{f})v,w\rangle=\int_G\langle\overline{f(g^{-1})}\pi(g)v,w\rangle dg=\int_G\langle v,f(g^{-1})\pi(g^{-1})w\rangle dg=\langle v,\pi(f)w\rangle$$ for $v,w\in H$, meaning that $\pi(\hat{f})=\pi(f)^\dag$. \\
\end{proof}

If $K<G$ is a compact subgroup, we can consider the subalgebra $C_c(K\backslash G/K,\mathbb{C})$ consisting of bi-$K$-invariant functions in $C_c(G,\mathbb{C})$, i.~e.~functions that satisfy $f(kgk')=f(g)$ for $g\in G,k,k'\in K$. It is easy to see that this condition is preserved under convolution and so we have indeed a subalgebra. \\

\begin{dfn}{(Gel'fand pair)}
Let $G$ be a locally compact group and let $K$ be a compact subgroup. The pair $(G,K)$ is a \textit{Gel'fand pair} if the convolution algebra $C_c(K\backslash G/K,\mathbb{C})$ is commutative, i.~e.~if $f_1*f_2=f_2*f_1$ for $f_1,f_2\in C_c(K\backslash G/K,\mathbb{C})$.
\end{dfn}

For more details about Gel'fand pairs and convolution algebras, a useful reference is \cite{vandijk}. \\

We now show how Gel'fand pairs appear in the setting of our paper.

\begin{prp} \label{gkgelfand}
Let $G$ be as in Theorem \ref{Main}, and let $K$ be the stabilizer of a vertex $o$. Then, $(G,K)$ is a Gel'fand pair.
\end{prp}

\begin{proof}
The proposition follows from the stronger fact that for every $g\in G$ we have $g^{-1}\in KgK$, as shown in (\cite{amann}, Proposition 7). To check this stronger fact, it suffices to observe that the set $KgK$ is the set of $h\in G$ such that $d(ho,o)=d(go,o)$, and we have $d(g^{-1}o,o)=d(o,go)$. \\
\end{proof}

We also recall that the groups $G$ of Theorem \ref{Main} are unimodular (\cite{amann}, Proposition 6); therefore, Lemma \ref{symmetric} applies. \\

We finally recall a result of Naimark that will be relevant for our classification:

\begin{prp} \label{naimark}
Let $\mathcal{A}$ be a commutative $\mathbb{C}$-algebra of bounded operators in a Pontryagin space $H$ of index $p\geq 1$ (that is, a Hilbert space equipped with a strongly nondegenerate sesquilinear form $\langle\cdot,\cdot\rangle$ of finite index $p$). Assume $\mathcal{A}$ is closed under taking adjoints for the sesquilinear form defining the space. Then, $\mathcal{A}$ preserves a finite-dimensional subspace $L$ that is non-negative, i.~e.~$\langle v,v\rangle\geq 0$ for $v\in L$.
\end{prp}

\begin{proof}
This is Corollary 2 in \cite{naimark1}. \\
\end{proof}

\section{Restricting to spherical representations}

In this section, we begin the classification of the representations. We will work on a group $G$ that satisfies the hypotheses of Theorem \ref{Main}, and on a topologically irreducible continuous representation $(\pi,H)$ of $G$ in a $\mathbb{K}$-Hilbert space $H$, such that the restriction $(\pi,H^{(\infty)})$ to the subspace of smooth vectors in $H$ is a smooth algebraically irreducible representation. Here $\mathbb{K}\in\{\mathbb{R},\mathbb{C},\mathbb{H}\}$, and in the case $\mathbb{K}=\mathbb{H}$, we will always intend $\mathbb{K}$-vector spaces as right vector spaces. We will also always assume all sesquilinear forms to be continuous and Hermitian. \\

\begin{dfn}{(Complete subtree)}
A finite subtree $S\subset T$ is \textit{complete} if every vertex in $S$ with at least two neighbors in $S$ is such that all its neighbors in $T$ are in $S$. \\
\end{dfn}

Given a continuous representation $(\pi,H)$ of $G$, we know from Proposition \ref{smooth} that there is some open compact $K<G$ such that $H^K\neq\{0\}$, where the notation indicates the subspace of $K$-fixed vectors in $H$. We can consider the set of complete finite subtrees $S$ such that $H^{G(S)}\neq\{0\}$. What we have said, together with Proposition \ref{props}(a), implies that this set is not empty. Thus, we can consider the set $M$ of minimal complete finite subtrees $S$ in the set and distinguish three cases:

\begin{itemize}
    \item $M$ contains a subtree that consists of only one vertex: we say that the representation $(\pi,H)$ is \textit{spherical}.
    \item $M$ contains a subtree that consists of two vertices (linked by an edge): we say that the representation $(\pi,H)$ is \textit{special}.
    \item $M$ contains a subtree with more than two vertices: we say that the representation $(\pi,H)$ is \textit{super-cuspidal}.
\end{itemize}

Since the set $M$ is $G$-invariant, it is clear that the representation $(\pi,H)$ belongs to exactly one of the three cases (it follows from (b) and (c) of Proposition \ref{props}). \\

We now show that every smooth algebraically irreducible representation of $G$ is a subrepresentation of a more concrete algebraic representation. Applying this to the representation $(\pi,H^{(\infty)})$ will allow us to restrict our analysis to spherical representations and express them in a more convenient way. \\

We define here this more concrete algebraic representation, the \textit{regular representation}, which depends on a complete finite subtree $S$. This representation was already introduced in \cite{olsh}.

\begin{dfn}{(Regular representation)}
Let $S$ be a complete finite subtree of $T$. We define the vector space $C(S)$ as the space of functions $f:G\rightarrow\mathbb{K}$ that are right-invariant under $G(S)$, i.~e.~$f(gh)=f(g)$ for every $g\in G$ and $h\in G(S)$, and left-invariant under some open compact subgroup of $G$, i.~e.~for every $f\in C(S)$ there exists an open compact $K_f<G$ such that $f(kg)=f(g)$ for every $g\in G$ and $k\in K_f$. The \textit{regular representation} $(\Pi_S,C(S))$ is an algebraic representation of $G$ defined as the natural left action on $C(S)$: for $g\in G$ and $f\in C(S)$, we have $(\Pi_S(g)f)(h)=f(g^{-1}h)$ for every $h\in G$. \\
\end{dfn}

We recall the definition of contragredient representation, which in our setting is the same as in \cite{bbej}.

\begin{dfn}{(Contragredient representation)}
Let $(\psi,Z)$ be a smooth algebraic representation of $G$. Consider the dual space $Z^*=\mathrm{Hom}_{\mathbb{K}}(Z,\mathbb{K})$. Every $\xi\in Z^*$ is a linear operator on $Z$, which we will denote by $\langle\xi,\cdot\rangle$. Consider now the algebraic representation $(\psi^*,Z^*)$ of $G$ given by $\langle\psi^*(g)\xi,v\rangle=\langle\xi,\psi(g^{-1})v\rangle$. The \textit{contragredient representation} $(\widetilde{\psi},\widetilde{Z})$ of $(\psi,Z)$ is given by the restriction of $(\psi^*,Z^*)$ to the subspace of smooth vectors in $Z^*$. \\
\end{dfn}

We are now ready to show that any smooth irreducible algebraic representation of $G$ is a subrepresentation of an appropriate regular representation.

\begin{prp} \label{subr}
Let $S$ be a complete finite subtree of $T$. Let $(\psi,Z)$ be a smooth irreducible algebraic representation of $G$ such that $Z^{G(S)}\neq\{0\}$. Then $(\psi,Z)$ is a subrepresentation of the regular representation $(\Pi_S,C(S))$.
\end{prp}

\begin{proof}
Consider the contragredient representation $(\widetilde{\psi},\widetilde{Z})$ of $(\psi,Z)$. Since $Z^{G(S)}\neq\{0\}$, we have $\widetilde{Z}^{G(S)}\neq\{0\}$: the space $Z^{G(S)}$ has a $G(S)$-invariant complement in $Z$ and therefore is isomorphic to $\widetilde{Z}^{G(S)}$ (\cite{bbej}, Proposition 2.7). \\

Consider a nonzero $\xi\in\widetilde{Z}^{G(S)}$. The map $$v\rightarrow\left(g\mapsto\left\langle \xi,\psi(g^{-1})v\right\rangle\right)$$ is a well-defined map from $Z$ to $C(S)$: the function $g\mapsto\left\langle \xi,\psi(g^{-1})v\right\rangle$ is right-invariant under $G(S)$ because $\xi$ is invariant under $G(S)$, and left-invariant under some compact open subgroup of $G$ because every $v\in Z$ is smooth. It is immediate to see that the map is $G$-equivariant and nonzero. Therefore, since the representation $(\psi,Z)$ is algebraically irreducible, the map is injective and $(\psi,Z)$ is a subrepresentation of $(\Pi_S,C(S))$. \\
\end{proof}

We apply Proposition \ref{subr} to the representation $(\pi,H^{(\infty)})$, choosing $S\in M$. We get that it is a subrepresentation of $(\Pi_S,C(S))$, and we might regard $H^{(\infty)}$ as a ($G$-invariant) subset of $C(S)$, with $\pi$ acting as the left action $\Pi_S$ on $C(S)$ restricted to $H^{(\infty)}$. Since $S\in M$, we have $\int_{G(U)}f(gk)dk=0$ for every $f\in H^{(\infty)}$ and for every proper complete subtree $U$ of $S$ (otherwise, there would be $U\subset S$ with a nonzero $G(U)$-invariant function in $H^{(\infty)}$, contradicting the minimality of $S$ among complete finite subtrees with this property). \\

Given the above properties, Amann proves (\cite{amann}, Lemma 20) that if the representation $(\pi,H)$ is super-cuspidal, the space $H^{(\infty)}\subseteq C(S)$ consists of functions that are finitely supported when projected to $G/S$. This is shown for complex representations, but the proof is the same for real and quaternionic ones. Amann then proceeds to classify unitary super-cuspidal representations, but the proof that the functions are finitely supported on $G/S$ works without assuming that the representation is unitary. This follows as a consequence, as the representation will preserve the scalar product on $H^{(\infty)}$ defined as the restriction to $H^{(\infty)}$ of the standard scalar product on the space of finitely supported functions on $G/S$. \\

Similarly, Amann proves (\cite{amann}, Lemma 34) that if the representation $(\pi,H)$ is special, all functions in $H^{(\infty)}\subset C(S)$ (with $S$ consisting of two vertices linked by an edge) belong to $\ell^2(G/S)$. Again, the proof does not need that the representation is unitary, and this fact follows as a consequence, as the restriction to $H^{(\infty)}$ of the standard scalar product on $\ell^2(G/S)$ will be preserved. \\

In the rest of the paper, we will focus our attention on the remaining case, the case of spherical representations. Since in the other two cases all algebraically irreducible representations are unitary and since they satisfy $\dim(H^{G(S)})<\infty$ (\cite{amann}, Lemma 20, Lemma 34), they cannot preserve a nondegenerate sesquilinear form of index $\geq 1$. This is explained in the following proposition. \\

\begin{prp} \label{uniqueform}
Let $(\psi,Z)$ be an algebraically irreducible $\mathbb{K}$-representation of $G$. Assume there exists an open compact subgroup $K$ such that $0<\dim(Z^K)<\infty$. If $\psi$ preserves a nondegenerate sesquilinear form $\langle\cdot,\cdot\rangle$ of index $p\geq 0$, with $2p\neq\dim(Z)$, then any other sesquilinear form $\langle\langle\cdot,\cdot\rangle\rangle$ that it preserves is a real scalar multiple of $\langle\cdot,\cdot\rangle$.
\end{prp}

\begin{proof}
Any Hermitian sesquilinear form $\langle\langle\cdot,\cdot\rangle\rangle$ on $Z$ can be written as $\langle\langle\cdot,\cdot\rangle\rangle=\langle A\cdot,\cdot\rangle$, where $A: Z\rightarrow Z$ is a self-adjoint linear operator. If the form $\langle\langle\cdot,\cdot\rangle\rangle$ is preserved by the representation $\psi$, we have $$\langle A\psi(g)v,w\rangle=\langle\langle\psi(g)v,w\rangle\rangle=\langle\langle v,\psi(g^{-1})w\rangle\rangle=\langle Av,\psi(g^{-1})w\rangle=\langle \psi(g)Av,w\rangle$$ for all $g\in G,\ v,w\in Z$. This means that $A$ commutes with the representation, i.~e.~$A\in\mathrm{End}(\psi)$. By the Schur Lemma (\cite{bbej}, Lemma 3.21), $\mathrm{End}(\psi)$ is a division algebra over $\mathbb{R}$. If we fix a nonzero $v\in Z^K$ the map $\mathrm{End}(\psi)\rightarrow Z^K$ defined by $\phi\mapsto\phi(v)$ is injective, since $\psi$ is algebraically irreducible, thus $\dim(\mathrm{End}(\psi))\leq\dim(Z^K)$. Since $\dim(Z^K)$ is finite, so is $\dim(\mathrm{End}(\psi))$, and it is well known that then $\mathrm{End}(\psi)$ must be isomorphic to $\mathbb{R}$, $\mathbb{C}$ or $\mathbb{H}$. \\

If $\mathrm{End}(\psi)$ is isomorphic to $\mathbb{R}$, then we are done. In the other two cases, we have $A=\alpha Id+\beta A_1$ with $\alpha,\beta\in\mathbb{R}$ and $A_1^2=-Id$. We wish to prove that $\beta=0$. Since $A$ and $\alpha Id$ are self-adjoint and belong to $\mathrm{End}(\psi)$, the same must be true for $\beta A_1$ and therefore, if $\beta\neq 0$, also for $A_1$. Then for every $v\in Z$ we have $$\langle A_1 v,A_1 v\rangle=\langle A_1^2 v,v\rangle=-\langle v,v\rangle.$$

Therefore, $A_1$ sends positive vectors (with respect to the form $\langle\cdot,\cdot\rangle$) to negative vectors and vice versa. As $A_1$ is invertible (because $A_1^2=-Id$), this means that it establishes a bijection between positive and negative vectors of $Z$, which can only happen if the index $p$ of the form $\langle\cdot,\cdot\rangle$ satisfies $2p=\dim(Z)$. \\
\end{proof}

\begin{rmk} \label{topuniqueform}
We have defined the notion of unitarity only for continuous representations in the Introduction. Here, when we talk about unitarity for algebraic representation, we simply mean that a positive-definite Hermitian sesquilinear form is preserved. As we have said in the Preliminaries, if $(\pi,H)$ is a continuous irreducible representation that preserves a nondegenerate sesquilinear form of finite index, then the restriction of the form to $H^{(\infty)}$ is still nondegenerate and of the same index. Therefore, if $(\pi,H)$ preserves a form of positive index, then $(\pi,H^{(\infty)})$ also preserves a form of positive index, and Proposition \ref{uniqueform} ensures that it cannot preserve a form of index $0$, hence $(\pi,H)$ cannot be super-cuspidal or special. \\
\end{rmk}

\section{An explicit form for spherical representations}

We start this section with an important proposition that will allow us to get an explicit form for spherical representations. The main property that allows for this result is that, if $G$ is as in Theorem \ref{Main} and $K<G$ is the stabilizer of a vertex, the pair $(G,K)$ is a Gel'fand pair (Proposition \ref{gkgelfand}).

\begin{prp} \label{gelfand}
Let $G$ be as in Theorem \ref{Main} and $K<G$ be the stabilizer of a vertex. For any topologically irreducible continuous representation $(\pi,H)$ of $G$ in a $\mathbb{K}$-Hilbert space $H$ (with $\mathbb{K}\in\{\mathbb{R},\mathbb{C},\mathbb{H}\}$), if the representation preserves a nondegenerate sesquilinear form of finite index $p$, then the space $H^K$ of $K$-fixed vectors is at most $1$-dimensional over $\mathbb{K}$.
\end{prp}

\begin{proof}
We consider the convolution algebra $\mathcal{A}=C_C(K\backslash G/K,\mathbb{K}')$, with $\mathbb{K}'=\mathbb{C}$ if $\mathbb{K}=\mathbb{C}$ and $\mathbb{K}'=\mathbb{R}$ if $\mathbb{K}\in\{\mathbb{R},\mathbb{H}\}$. As we have recalled in the Preliminaries, this is the algebra of continuous compactly supported $\mathbb{K}'$-valued functions on $G$ which are invariant under $K$ on the left and on the right. \\

To prove our proposition, we follow the proof of Proposition 5.4 in \cite{monodpy}. \\

We know that the algebra $\mathcal{A}$ is commutative since $(G,K)$ is a Gel'fand pair, and \textit{symmetric}, meaning that the adjoint of every element of $\mathcal{A}$ (with respect to the sesquilinear form preserved by $\pi$) is again in $\mathcal{A}$ (as in Lemma \ref{symmetric}). Since in our case $g^{-1}\in KgK$ for every $g\in G$ (as shown in the proof of Proposition \ref{gkgelfand}), a stronger statement holds: the adjoint of every $f\in\mathcal{A}$ is its complex conjugate $\overline{f}$ (or $f$ itself, if $\mathbb{K}'=\mathbb{R}$). In fact, we see from the proof of Lemma \ref{symmetric} that the adjoint of $f\in C_c(G,\mathbb{K}')$ is the function $\hat{f}(g)=\overline{f(g^{-1})}$, and if $g^{-1}\in KgK$ and $f$ is left- and right-$K$-invariant this means $\hat{f}(g)=\overline{f(g)}$. \\

As in \cite{monodpy} we see that $\mathcal{A}$ acts irreducibly on $H^K$ (which we assume to be non-trivial) and that the restriction of the sesquilinear form on $H$ to $H^K$ is non-degenerate (say of index $q\geq 0$). We first assume $q\geq 1$, so that we can apply Proposition \ref{naimark}. \\

In the case $\mathbb{K}=\mathbb{C}$, Proposition \ref{naimark} applied to the representation of $\mathcal{A}$ on $H^K$ ensures that $\mathcal{A}$ preserves a finite-dimensional subspace $L$ of $H^K$, and as the representation is irreducible it must be $L=H^K$. Since $\mathcal{A}$ is commutative, every element of $\mathcal{A}$ is an endomorphism of the representation, so the Schur Lemma implies that $\mathcal{A}\cong\mathbb{C}$ and therefore $H^K$ is $1$-dimensional over $\mathbb{C}$. \\

In the case $\mathbb{K}=\mathbb{R}$, we apply Proposition \ref{naimark} to the complexification of the representation of $\mathcal{A}$ on $H^K$ (defined as $H^K\otimes_\mathbb{R} \mathbb{C}$). We find that there exists a finite-dimensional complex vector space $L$ preserved under $\mathcal{A}\otimes\mathbb{C}$ and thus under $\mathcal{A}$. This implies that there exists a finite-dimensional real vector space $L'\subseteq H^K$ preserved under $\mathcal{A}$, which must coincide with $H^K$ by irreducibility. The representation preserves a positive definite bilinear form on $L'=H^K$, since the restriction of the sesquilinear form to $L'$ is non-negative and non-degenerate. If we apply the Schur Lemma, we see that $\mathcal{A}$ is a finite-dimensional commutative division algebra over $\mathbb{R}$ and therefore it is isomorphic to $\mathbb{R}$ or $\mathbb{C}$. In the latter case, there exists $A\in\mathcal{A}$ such that $A^2=-Id$, which implies, since each $A\in\mathcal{A}$ is self-adjoint, that for every $v\in H^K$ it holds $$\langle Av,Av\rangle=\langle A^2 v,v\rangle=-\langle v,v\rangle,$$ contradicting the fact that the form is positive definite. Therefore, it must be $\mathcal{A}\cong\mathbb{R}$, which implies that $H^K$ is $1$-dimensional over $\mathbb{R}$. \\

In the case $\mathbb{K}=\mathbb{H}$, we apply Proposition \ref{naimark} to the complex representation of $\mathcal{A}\otimes\mathbb{C}$ on $H^K$ seen as a complex vector space and find that $\mathcal{A}\otimes\mathbb{C}$ and hence $\mathcal{A}$ preserves a finite-dimensional complex subspace $L$. The quaternionic vector space that it generates is finite-dimensional, it must coincide with $H^K$ by irreducibility, and the sesquilinear form on it is positive definite. Applying the Schur Lemma, we again see that $\mathcal{A}\cong\mathbb{R}$ or $\mathcal{A}\cong\mathbb{C}$. The case $\mathcal{A}\cong\mathbb{C}$ can be excluded with the same argument as before, and we can conclude that $H^K$ is $1$-dimensional over $\mathbb{H}$. \\

We are left with the case $q=0$, in which the restriction of the sesquilinear form to $H^K$ is positive definite and hence $H^K$ equipped with the form is a Hilbert space, on which $\mathcal{A}$ acts irreducibly. We can now proceed as in the addendum to \cite{monodpy}. If $\mathbb{K}=\mathbb{C}$, the usual Schur Lemma for operator algebras applies directly and yields $\dim_{\mathbb{C}}(H^K)\leq 1$. If $\mathbb{K}=\mathbb{R}$, we can use the real version of the Schur Lemma, which states that the commutant of $\mathcal{A}$ is isomorphic to $\mathbb{R}$, $\mathbb{C}$ or $\mathbb{H}$, and allows to conclude (since $\mathcal{A}$ is commutative) that $\mathcal{A}$ and hence $H^K$ is finite-dimensional, like in the addendum to \cite{monodpy}. Then we can conclude $\dim_{\mathbb{R}}(H^K)\leq 1$ as we did in the case $q\geq 1$. Finally, if $\mathbb{K}=\mathbb{H}$, we need a quaternionic version of the Schur Lemma, which can be found, for instance, in \cite{quantum}. Since in our case every element of $\mathcal{A}$ is self-adjoint, we can use Proposition 5.3 in \cite{quantum} to directly obtain that every operator in $\mathcal{A}$ is a real scalar, which forces $\dim_{\mathbb{H}}(H^K)\leq 1$ by irreducibility. \\
\end{proof}

We can now proceed to classify spherical representations. \\

We can apply Proposition \ref{subr} with $S$ being a single vertex. We call this vertex $o$, and we set $K:=G(o)$. \\

The space $C(S)$ is the space of $K$-right-invariant functions on $G$ that are left-invariant under some open compact subgroup. We can identify this space with the space of functions on $G/K$ (the set of right cosets of $G$ under $K$) that are left-invariant under some compact open subgroup of $G$. The set $G/K$ can be identified with the orbit of $o$ in $T$ under the action of $G$, which is either the full vertex set $V$ or the set of vertices at even distance from $o$ (Proposition \ref{props}(c)). \\

For the rest of this section and the next one, we assume that we have the full vertex set $V$, i.~e.~that the action of $G$ is transitive on the vertices. In this case, the tree must be regular. In Section 6 we will extend the results to the other case. \\

Proposition \ref{subr} tells us that we can see the representation $(\pi,H^{(\infty)})$ as a subrepresentation of $(\Pi,C)$, where $C$ is the space of $\mathbb{K}$-valued functions on $V$ that are left-invariant under some open compact subgroup of $G$ and $\Pi$ is the left action of $G$ on $C$. Therefore, from now on, we assume that $H^{(\infty)}\subset C$ and $\pi$ coincides with $\Pi$ on $H^{(\infty)}$. \\

Proposition \ref{gelfand} implies that $\dim_{\mathbb{K}}(H^K)=1$. Since $H^K\subseteq H^{(\infty)}$, this means that there exists a nonzero function $f\in H^{(\infty)}\subset C$ such that the space of $K$-invariant functions in $H^{(\infty)}\subset C$ is the space of scalar multiples of $f$. \\

\begin{obs}
A function in $C$ is $K$-invariant if and only if it is \textit{radial} with respect to the base vertex $o$, i.~e.~its value at a vertex $v$ only depends on the distance between $v$ and $o$. This follows from the fact that $K$ acts transitively on $\partial T$ (Proposition 2.2(b)). Therefore, we can use the notation $f(n)$ ($n\in\mathbb{N}$) to indicate the value of $f$ at any vertex at distance $n$ from $o$. \\
\end{obs}

\begin{dfn}{(Laplacian operator)} \label{lapl}
The \textit{Laplacian operator} on the tree $T$ is the operator $L: C\rightarrow C$ defined as follows: if $f: V\rightarrow\mathbb{K}$ is a function in $C$, the function $Lf$ associates to every vertex $v$ in $V$ the average of the values of $f$ at neighbors of $v$. \\
\end{dfn}

\begin{obs} \label{calclapl}
The Laplacian $Lf$ of a radial function $f$ is a radial function, with the following values ($d$ is the degree of the tree):
\begin{itemize}
    \item $Lf(0)=f(1)$,
    \item $Lf(n)=\frac{1}{d}f(n-1)+\frac{d-1}{d}f(n+1)$ for all $n\geq 1$. \\
\end{itemize}
\end{obs}

\begin{prp}
The function $f$ is an eigenfunction of the Laplacian, i.~e.~there exists some $\alpha\in\mathbb{K}$ such that $Lf=f\cdot\alpha$.
\end{prp}

\begin{proof}
Since the action of $G$ on $V$ is transitive, for every neighbor $w$ of $o$ there is an element $g\in G$ such that $go=w$. If we consider the function $\pi(g)f$, it is a radial function with respect to $w$: for every $n$, the value of $\pi(g)f$ at any vertex at distance $n$ from $w$ is equal to $f(n)$. \\

We repeat the process for all neighbors of $o$, and we average the functions $\pi(g)f$ that we obtain. We get a function $Af\in H^{(\infty)}$, radial with respect to $o$. We have:
\begin{itemize}
    \item $Af(0)=f(1)$,
    \item $Af(n)=\frac{1}{d}f(n-1)+\frac{d-1}{d}f(n+1)$ for all $n\geq 1$.
\end{itemize}
Therefore, due to Observation \ref{calclapl}, $Af=Lf$. \\

Since the space of radial functions in $H^{(\infty)}$ (which is the same as the space of $K$-invariant functions in $H^{(\infty)}$) is $1$-dimensional and both $f$ and $Af$ are in $H^{(\infty)}$, we must have $Af=f\cdot\alpha$, and hence $Lf=f\cdot\alpha$, for some $\alpha\in\mathbb{K}$ (recall that we intend the scalar multiplication to be on the right if $\mathbb{K}=\mathbb{H}$). \\
\end{proof}

It is immediate to see, from Observation \ref{calclapl}, that the radial function $f$ is uniquely determined by $f(0)$ and by the property $Lf=f\cdot\alpha$. In particular, if $f(0)=0$ then $f$ is identically zero; therefore, we might assume $f(0)\neq 0$, and up to multiplying by a scalar $f(0)=1$. We can then calculate $f(1)=\alpha$ and $f(2)=\frac{1}{d-1}(df(1)\cdot\alpha-f(0))=\frac{(\alpha^2 d)-1}{d-1}$. In general, $f(n)$ is a polynomial in $\alpha$ with real coefficients for every $n$. This implies that even in the case $\mathbb{K}=\mathbb{H}$, $\alpha$ commutes with $f(n)$ for every $n$, therefore $f$ satisfies $Lf=\alpha\cdot f$, where the notation indicates that each value of $f$ is multiplied by $\alpha$ on the left. This is particularly convenient since the space $\{h\in C: Lh=\alpha\cdot h\}$ is a right $\mathbb{K}$-vector subspace of $C$.\\

If $\alpha=1$, the function $f$ is constant on $V$ and thus $G$-invariant. Hence $H^{(\infty)}$ is generated by $f$, the representation $(\pi,H^{(\infty)})$ is the trivial representation, and from this it follows that $(\pi,H)$ is also the trivial representation. If $\alpha=-1$, the function $f$ has the same value $f(0)$ at every vertex at even distance from $o$ and the value $-f(0)$ at every vertex at odd distance from $o$. An element of $G$ sends $f$ to itself if it moves vertices by an even distance, and to $-f$ otherwise. The space $H^{(\infty)}$ is again $1$-dimensional, and the representation $(\pi,H^{(\infty)})$ is the $1$-dimensional character of order $2$; its kernel is the subgroup that moves vertices by an even distance. Again, like in the previous case, the representation $(\pi,H)$ coincides with $(\pi,H^{(\infty)})$. \\

We are left with the general case $\alpha\neq\pm 1$. This implies $f(2)\neq f(0)$: otherwise, we would have $(\alpha ^2 d)-1=d-1$ and thus $\alpha=\pm 1$. \\

We are ready to prove a major result of this paper, which offers a form of classification of algebraically irreducible spherical representations of $G$. \\

\begin{thm} \label{sph}
For any algebraically irreducible spherical representation $(\pi,H^{(\infty)})$ of $G$ that is not $1$-dimensional, there exists $\alpha\in\mathbb{K}$, $\alpha\neq\pm 1$, such that the representation is equivalent to the left action of $G$ on the subspace $H_{\alpha}$ of $C$ defined by $H_{\alpha}:=\{h\in C: Lh=\alpha\cdot h\}$. Conversely, for every $\alpha\in\mathbb{K}$, $\alpha\neq\pm 1$, the algebraic representation of $G$ defined as the natural left action on the space $H_{\alpha}$ is algebraically irreducible and spherical.

(Note: It is not guaranteed that all these representations preserve a finite-index sesquilinear form; in fact, we will prove later that this is only possible if $\alpha\in\mathbb{R}$)
\end{thm}

\begin{proof}
We start by picking a general algebraically irreducible spherical representation of $G$ and using the above results. We know that there exist $\alpha\in\mathbb{K}$ and a radial $f\in H^{(\infty)}\subset C$ with $f(0)=1$ such that $Lf=\alpha\cdot f$ and the subspace of $K$-invariant functions in $H^{(\infty)}$ is the space of scalar multiples of $f$. We can assume $\alpha\neq\pm 1$, and thus $f(2)\neq f(0)$: otherwise, as we have seen, the representation is $1$-dimensional. \\

Since the representation is algebraically irreducible, $H^{(\infty)}$ is generated (as a vector space) by $\{\pi(g)f: g\in G\}$. A function $\pi(g)f$ is radial with respect to $go$ and satisfies $L\pi(g)f=\alpha\cdot \pi(g)f$ (it is easily seen that both the Laplacian and the left multiplication by $\alpha$ commute with the action of $G$). As the property $Lf=\alpha\cdot f$ is preserved under linear combinations, we have $H^{(\infty)}\subseteq H_{\alpha}$. \\

To obtain the other inclusion, we pick a generic $h\in H_{\alpha}$, that is, a generic $h\in C$ such that $Lh=\alpha\cdot h$, and show that it is in $H^{(\infty)}$. Since $H^{(\infty)}$ is $G$-invariant and $f\in H^{(\infty)}$, it is enough to show that we can express $h$ as a finite linear combination of translates (by elements of $G$) of $f$. \\

For every $v\in V$, let $N(v)$ be the set of neighbors of $v$ that are further away from the origin than $v$ (if $v=o$ we take the full set of neighbors). To build $h$ as a finite linear combination of translates of $f$, we start with the function $F=f\cdot h(o)$, which agrees with $h$ at the origin (recall that $f(0)=1$). As $LF(o)=\alpha\cdot F(o)=\alpha\cdot h(o)=Lh(o)$, the sum of $F$ over $N(o)$ agrees with $h$. To make the two functions agree at every vertex in $N(o)$, we add to $F$ the function $$\sum_{w\in N(o)}{\pi(g_w)f\cdot(f(0)-f(2))^{-1}(h(w)-F(w))},$$ where $g_w\in G$ sends $o$ to $w$. The sum of the coefficients is zero and this guarantees that $F$ is unchanged at $o$. At every $w\in N(o)$ the function changes by $$f(0)(f(0)-f(2))^{-1}(h(w)-F(w))$$ due to the contribution from $\pi(g_w)f$, and by $$\sum_{z\in N(o)\setminus\{w\}}{f(2)(f(0)-f(2))^{-1}(h(z)-F(z))}=-f(2)(f(0)-f(2))^{-1}(h(w)-F(w))$$ due to the contribution from $\pi(g_z)f$ for $z\neq w$. This adds up to $h(w)-F(w)$, so we have a new function $F$ that agrees with $h$ on the ball $B_1(o)$. \\

We can move forward by applying the same procedure to each of the sets $N(w)$ with $w\in N(o)$: this will allow to update the function $F$ so that it agrees with $h$ inside $N(w)$ without changing the values inside the ball $B_1(o)$, and ultimately get a function that agrees with $h$ on the ball $B_2(o)$. We can continue inductively with $B_3(o)$ and so on. After finitely many steps $F$ and $h$ will agree everywhere: $h$ is left-invariant under some open compact subgroup that contains $G(B_n(o))$ for some $n$ (Proposition \ref{props}(a)), so it is uniquely determined by the property $Lh=\alpha\cdot h$ and the values inside $B_n(o)$, which means that it suffices to have $F$ and $h$ agree on $B_n(o)$. \\

We have shown that the space $H^{(\infty)}$ coincides with $H_{\alpha}$, so the representation from which we started is equivalent to the left action of $G$ on that space. \\

For the converse, we wish to prove that if we start from a generic $\alpha\in\mathbb{K}$, $\alpha\neq\pm 1$, and consider the space $H_{\alpha}=\{h\in C: Lh=\alpha\cdot h\}$ with the left action of $G$, we have a spherical algebraically irreducible representation. We know that there is a radial (hence $K$-invariant) nonzero function in this space, so we are only left to prove algebraic irreducibility. Suppose there is a nonzero $G$-invariant subspace $X\subset H_{\alpha}$ and pick $h\in X$, $h\neq 0$. It suffices to prove that any function in $H_{\alpha}$ can be obtained as a finite linear combination of $G$-translates of $h$. \\

We have seen before that every function in $H_{\alpha}$ is a finite linear combination of $G$-translates of $f$, therefore it suffices to show that $f$ is a finite linear combination of $G$-translates of $h$. Up to translating, we may assume $h(o)\neq 0$. If we average $h$ over $K$ (with respect to the Haar measure), we get a radial function that has to be a nonzero scalar multiple of $f$. The operation of averaging over $K$ is the same as applying a finite linear combination of $G$-translates (actually $K$-translates), since the function $h$ is left-invariant under $G(B_n(o))$ for some $n$ and the quotient $K/G(B_n(o))=G(o)/G(B_n(o))$ is finite. \\
\end{proof}

\section{Preserved sesquilinear forms}

In this section, we finish the proof of Theorem \ref{Main} in the case where $G$ acts transitively on $V$. In order to do this, we study nondegenerate sesquilinear forms of finite index that can be preserved by the spherical representations that we have classified in Theorem \ref{sph}. \\

We first recall that all super-cuspidal and special representations are unitary (this was explained in the discussion following Proposition \ref{subr}), which implies that, due to Proposition \ref{uniqueform}, such representations cannot preserve any nondegenerate sesquilinear form of index $\geq 1$ (see also Remark \ref{topuniqueform}). Therefore, to prove Theorem \ref{Main}, it suffices to consider the case of spherical representations. \\

Therefore, let $(\pi,H)$ be an irreducible spherical representation of $G$ that is not $1$-dimensional. We apply Theorem \ref{sph} and find that there exists $\alpha\in\mathbb{K}$, $\alpha\neq\pm 1$, such that $H^{(\infty)}$ is the space of functions $h: V\rightarrow\mathbb{K}$ that are left-invariant under some open compact subgroup of $G$ and satisfy $Lh=\alpha\cdot h$, and $\pi$ is the natural left action of $G$ on this space. \\

We split the space $H^{(\infty)}$ into a direct sum of $K$-invariant subspaces. First, we define $H_0$ as the $1$-dimensional space generated by the unique (up to scalar multiple) $K$-invariant $f\in H^{(\infty)}$. For $v\in V$, we define the set $N(v)$ as in the proof of Theorem \ref{sph}: $N(v)$ is the set of neighbors of $v$ that are further away from the origin than $v$ (or the complete set of neighbors if $v=o$). We define the space $\widetilde{H}_1$ as the subspace of $H^{(\infty)}$ consisting of functions that are invariant under $G(B_1(o))$, and then $H_1$ such that $\widetilde{H}_1=H_0\oplus H_1$, in this way: $$H_1:=\left\{h\in\widetilde{H}_1: h(o)=0,\ \sum_{v\in N(o)}{h(v)}=0\right\}.$$

We can proceed inductively. We define, for $n\in\mathbb{N}$, the set $S_n:=\{v\in V: d(o,v)=n\}$. Then we define, for $n\geq 2$, the space $\widetilde{H}_n$ of functions in $H^{(\infty)}$ that are invariant under $G(B_n(o))$, and then $$H_n:=\left\{h\in\widetilde{H}_n: h\vert_{B_{n-1}(o)}\equiv 0,\ \sum_{w\in N(v)}{h(w)}=0\ \forall v\in S_{n-1}\right\},$$ so that $\widetilde{H}_n=\bigoplus_{i=0}^{n}{H_i}$. \\

The space $H^{(\infty)}$ is the increasing union of $\widetilde{H}_n$, and hence is the direct sum of the spaces $H_n$. All these spaces are clearly $K$-invariant. We have $\dim(H_0)=1$, $\dim(H_1)=d-1$, and $\dim(H_n)=d(d-2)(d-1)^{n-2}$ for all $n\geq 2$. \\

It is clear that $H_0$ is $K$-irreducible, since it is $1$-dimensional. We show that $H_1$ is also $K$-irreducible. Since $K$ acts transitively on $\partial T$, it also acts transitively on $S_1$. Moreover, if $v\in S_1$ the stabilizer of $v$ in $G$ acts transitively on $\partial T$. This implies that its intersection with $K=G(o)$ acts transitively on $S_1\setminus\{v\}$, which (since $v$ is generic) is the same as saying that the action of $K$ on $S_1$ is doubly transitive. Now, let $X$ be any nonzero $K$-invariant subspace of $H_1$ and let $h\in X$ be nonzero. By averaging $h$ over $K\cap G(v)$, where $v\in S_1$ satisfies $h(v)\neq 0$, we get a nonzero scalar multiple of a function with value $1$ at $v$ and $-\frac{1}{d-1}$ at all other vertices in $S_1$. Since $K$-translates of this function span $H_1$, we must have $X=H_1$. \\

For spaces $H_n$ with $n\geq 2$ we cannot prove $K$-irreducibility, but we can prove the following:

\begin{lem} \label{dimbound}
For every $n\geq 2$, each $K$-irreducible subspace of $H_n$ has dimension at least $\vert S_{n-1}\vert=d(d-1)^{n-2}$.
\end{lem}

\begin{proof}
For $v\in V$, $v\neq o$, we define the \textit{cone} $C_v$ as the set of vertices $w$ such that the unique path between $o$ and $w$ passes through $v$. We define the \textit{anti-cone} $D_v:=(V\setminus C_v)\cup\{v\}$. As $G$ has the independence property, it follows from Proposition 11 of \cite{amann} (applied to the subtree $B_{n-1}(o)$) that $$G(B_{n-1}(o))=\prod_{v\in S_{n-1}}G(D_v).$$ Clearly, all factors commute. If $X$ is a $K$-irreducible subspace of $H_n$ and $h\in X$ is nonzero, we can choose $v\in S_{n-1}$ such that $h(w)\neq 0$ for at least one $w\in N(v)$ (otherwise, $h$ would vanish on $S_n$ and hence everywhere since it is in $H_n$). We average $h$ across $\prod_{z\in S_{n-1}\setminus\{v\}}G(D_z)$ and get a function that is unchanged on $N(v)$ and zero on $N(z)$ for every $z\in S_{n-1}\setminus\{v\}$. By the transitivity of $K$ on $S_{n-1}$ we can then get such a function for any $v\in S_{n-1}$. In this way, we get $\vert S_{n-1}\vert$ linearly independent functions in $X$, which proves our claim. \\
\end{proof}

Since $\vert S_{n-1}\vert=d(d-1)^{n-2}>\dim{H_{n-1}}$ for every $n\geq 2$, the lemma implies the following:

\begin{cor} \label{dec}
Let $i,j\in\mathbb{N}$ with $i\neq j$. Then, no $K$-irreducible subspace of $H_i$ can be isomorphic to any $K$-irreducible subspace of $H_j$. \\
\end{cor}

To continue the proof of Theorem \ref{Main}, we fix a $G$-invariant nondegenerate sesquilinear form $B$ on $H^{(\infty)}$. It follows from the Schur Lemma and Corollary \ref{dec} that $B$ must be the sum (for $n\in\mathbb{N}$) of $K$-invariant sesquilinear forms $B_n$ on $H_n$. This implies that the direct sum $H^{(\infty)}=\bigoplus_{n\in\mathbb{N}}{H_n}$ is orthogonal for the form $B$. We can now show the following:

\begin{lem} \label{alphar}
If the representation preserves the nondegenerate sesquilinear form $B$ (which we assume to be antilinear in the first variable and linear in the second), then $\alpha\in\mathbb{R}$.
\end{lem}

\begin{proof}
Let $f$ be the unique function in $H_0$ satisfying $f(0)=1$. It satisfies $f(1)=\alpha$ and $f(2)=\frac{(\alpha^2 d)-1}{d-1}$. Let $h$ be a function in $H_1$ satisfying $h(v)=1$ for one particular $v\in N(o)$ and $h(w)=-\frac{1}{d-1}$ for all $w\in N(o)\setminus\{v\}$. Since $f\in H_0$ and $h\in H_1$, we have $B(f,h)=0$. \\

If we consider $g\in G$ that sends $o$ to $v$ and $v$ to $o$, we have $\pi(g)f=f\cdot\alpha+h\cdot(1-\alpha^2)$ and $\pi(g)h=f-h\cdot\alpha$. The latter implies $$B(f,f)=B(h,h)\left(1-\vert\alpha\vert^2\right).$$ Observe that $B(f,f)\neq 0$, because otherwise the form would be degenerate ($f$ is orthogonal to all functions in $\bigoplus_{n\geq 1}{H_n}$). Hence, the last relation implies $B(h,h)\neq 0$. \\

As $B(f,h)=0$ it must be $B(\pi(g)f,\pi(g)h)=0$. This means that $B(f\cdot\alpha+h\cdot(1-\alpha^2),f-h\cdot\alpha)=0$, and therefore $$B(f,f)\overline{\alpha}-B(h,h)(1-\overline{\alpha}^2)\alpha=0.$$ Combining this with the expression we got before for $B(f,f)$ we get $$B(h,h)\left(1-\vert\alpha\vert^2\right)\overline{\alpha}=B(h,h)(1-\overline{\alpha}^2)\alpha.$$ We can divide by $B(h,h)$ on the left as it is nonzero. By developing the two expressions we get $$\overline{\alpha}-\vert\alpha\vert^2\overline{\alpha}=\alpha-\overline{\alpha}\vert\alpha\vert^2$$ and, therefore, $\overline{\alpha}=\alpha$, meaning that $\alpha\in\mathbb{R}$. \\
\end{proof}

\begin{cor} \label{tensor}
Every spherical algebraically irreducible complex or quaternionic representation of $G$ preserving a nondegenerate sesquilinear form is obtained by tensoring with $\mathbb{C}$ or $\mathbb{H}$ a spherical algebraically irreducible real representation of $G$ preserving a nondegenerate bilinear form. \\
\end{cor}

\begin{proof}
Due to Theorem \ref{sph}, any spherical algebraically irreducible representation is either $1$-dimensional or equivalent to the left action of $G$ on the space $H_{\alpha}=\{h\in C: Lh=\alpha\cdot h\}$ for some $\alpha\in\mathbb{K}$, $\alpha\neq\pm 1$, where $C$ is the space of functions $h: V\rightarrow\mathbb{K}$ that are left-invariant under some open compact subgroup of $G$. In the context of this proof, to emphasize the dependence on the field $\mathbb{K}$, we might use the notation $H_{\alpha}^{\mathbb{K}}$ for $H_{\alpha}$. \\

The $1$-dimensional case is trivial (we have either the trivial representation or a character of order $2$, as discussed before Theorem \ref{sph}). In the non-trivial case, if the representation preserves a nondegenerate sesquilinear form, Lemma \ref{alphar} implies that $\alpha\in\mathbb{R}$. In that case, it is clear that $H_{\alpha}^{\mathbb{C}}=H_{\alpha}^{\mathbb{R}}\otimes_{\mathbb{R}}\mathbb{C}$ and $H_{\alpha}^{\mathbb{H}}=H_{\alpha}^{\mathbb{R}}\otimes_{\mathbb{R}}\mathbb{H}$; moreover, the left action of $G$ on $H_{\alpha}^{\mathbb{R}}$ extends naturally to the tensor product with $\mathbb{C}$ or $\mathbb{H}$. Therefore, if we start from a complex or quaternionic representation that corresponds to some $\alpha\in\mathbb{R}$, $\alpha\neq\pm 1$, we have that it is the tensor product with $\mathbb{C}$ or $\mathbb{H}$ of the real representation corresponding to the same $\alpha$, which is spherical and algebraically irreducible due to the converse part of Theorem \ref{sph}, and preserves the bilinear form obtained by restricting the original sesquilinear form to the subspace of real functions in $H_{\alpha}^{\mathbb{C}}$ or $H_{\alpha}^{\mathbb{H}}$. \\
\end{proof}

Corollary \ref{tensor} also applies to spherical topologically irreducible continuous representations $(\pi,H)$ that preserve a strongly nondegenerate sesquilinear form of finite index, since in that case the space $H$ is the completion of $H^{(\infty)}$ for the Hilbert norm associated with the sesquilinear form, as in the definition of strongly nondegenerate forms in the Introduction. This proves the statement about $p=1$ in Theorem \ref{Main}, and allows to reduce to the case $\mathbb{K}=\mathbb{R}$ for the rest of the theorem. \\

Therefore, from now on, we may assume $\mathbb{K}=\mathbb{R}$, and we are left to show that no nondegenerate bilinear form of finite index $>1$ can be preserved by a spherical algebraically irreducible representation of $G$. \\

The strategy will be to present, for every $\alpha\in\mathbb{R}\setminus\{\pm 1\}$, an explicit nondegenerate bilinear form of index $0$ or $1$ that is preserved by the corresponding representation of $G$. This will exclude, by Proposition \ref{uniqueform}, that the representation can preserve a nondegenerate bilinear form of index $>1$, proving the theorem. Furthermore, this will provide a classification of nondegenerate bilinear forms that can be preserved by a spherical irreducible representation. As we have said, from now on we work with real representations, but the same results will hold for complex and quaternionic representations. \\

Recall that for every $n\geq 1$, a function in $H_n$ is uniquely determined by its values on the set of vertices $S_n=\{v\in V: d(o,v)=n\}$. More precisely, each function $S_n\rightarrow\mathbb{R}$ with sum $0$ on the set of neighbors of any vertex in $S_{n-1}$ can be extended in a unique way to a function in $H_n$. \\

We observe that if we define the bilinear form $Q_n$ on $H_n$ ($n\geq 1$) so that it coincides with the restriction of the standard scalar product on the space of functions $S_n\rightarrow\mathbb{R}$ to the subspace consisting of functions with sum $0$ on the set of neighbors of any vertex in $S_{n-1}$, then it is preserved by the action of $K=G(o)$. \\

Let $Q_0$ be the bilinear form on $H_0$ defined by $Q_0(f,f)=1$ (where $f\in H_0$ satisfies $f(0)=1$). We define the bilinear form $Q$ on $H^{(\infty)}$ as follows: $$Q:=\frac{d}{d-1}(1-\alpha^2)Q_0+\sum_{n\geq 1}{Q_n}.$$

This bilinear form is nondegenerate with index $0$ if $\vert\alpha\vert<1$ and index $1$ if $\vert\alpha\vert>1$. \\

To finish the proof of Theorem \ref{Main} in the case where $G$ acts transitively on $V$, it suffices to show the following:

\begin{prp} \label{qpres}
The bilinear form $Q$ is preserved under $G$. \\
\end{prp}

It is clear from the definition that $Q$ is preserved under $K$. \\

Fix a neighbor $v$ of $o$ and consider $g\in G$ that sends $o$ to $v$ and $v$ to $o$. \\

\begin{lem} \label{gkgen}
The element $g$ and the subgroup $K$ generate the group $G$.
\end{lem}

\begin{proof}
Let $g_0\in G$. To show that $g_0$ can be generated with $g$ and $K$, let $w:=g_0(o)$. If we find $g_1$ generated with $g$ and $K$ such that $g_1(o)=w$, we will have $g_1^{-1}g_0(o)=o$ and therefore $g_1^{-1}g_0\in K$, concluding. \\

To find such a $g_1$, let $m:=d(o,w)$. Observe that by composing $g$ with some element of $K$ we can get an element $g'$ that sends $o$ to $v$ and $v$ to a vertex in $S_2$. Then the element $(g')^m$ will send $o$ to a vertex in $S_m$, and by composing with an element of $K$ we can get $g_1$ such that $g_1(o)=w$. \\
\end{proof}

The lemma implies that, to prove Proposition \ref{qpres}, it suffices to show that $Q$ is preserved under $g$. \\

Define the functions $f\in H_0$ and $h\in H_1$ as in the proof of Lemma \ref{alphar}. From the definition of $Q$ we have $Q(f,f)=\frac{d}{d-1}(1-\alpha^2)$ and $Q(h,h)=\frac{d}{d-1}$, since the $d$-dimensional vector $(1,-\frac{1}{d-1},\cdots,-\frac{1}{d-1})$ has Euclidean norm $\frac{d}{d-1}$. Hence $Q(f,f)=(1-\alpha^2)Q(h,h)$, and therefore $g$ preserves the norm of $h$ and the orthogonality of $f$ and $h$: the two conditions were explicitated in the proof of Lemma \ref{alphar}, and it is immediate to see that both hold if $\alpha\in\mathbb{R}$ and $B(f,f)=(1-\alpha^2)B(h,h)$. The norm of $f$ is also preserved, because $\pi(g)f=\alpha f+(1-\alpha^2)h$ and the norm of this function is $$\alpha^2 Q(f,f)+(1-\alpha^2)^2 Q(h,h)=(\alpha^2+(1-\alpha^2))Q(f,f)=Q(f,f).$$

We have shown that $g$ preserves $Q$ in $\mathrm{Span}\{f,h\}$. Observe now that the orthogonal complement of $\mathrm{Span}\{f,h\}$ in $H^{(\infty)}$ is the space of functions in $H^{(\infty)}$ with value $0$ at $o$ and $v$. In fact, it is clear that all such functions are orthogonal to $f$ and $h$, and to see that they generate $H^{(\infty)}$ together with $f$ and $h$, it suffices to take any $t\in H^{(\infty)}$ and subtract $t(o)f+(t(v)-\alpha t(o))h$ to obtain a function that vanishes at $o$ and $v$. Furthermore, since $g$ swaps $o$ and $v$, the space $\mathrm{Span}\{f,h\}^\perp$ is invariant under $g$. \\

Therefore, to complete the proof of Proposition \ref{qpres}, it suffices to show that $g$ preserves the norm of every function in $\mathrm{Span}\{f,h\}^\perp$. To see this, we further split each $H_n$ ($n\geq 2$) into two orthogonal subspaces: $H_n=H_n^o\oplus H_n^v$, where $H_n^o$ is the space of functions in $H_n$ supported on $\{w\in V: d(o,w)<d(v,w)\}$, and $H_n^v$ is the space of functions in $H_n$ supported on $\{w\in V: d(o,w)>d(v,w)\}$. We can also split $H_1$ as $H_1=H_1^o\oplus\mathrm{Span}\{h\}$. For every $j\in\mathrm{Span}\{f,h\}^\perp$, we call $j_n^o$ its projection to $H_n^o$ and $j_n^v$ its projection to $H_n^v$. We have $$\pi(g)j=\pi(g)\left(\sum_{n\geq 1}{j_n^o}+\sum_{n\geq 2}{j_n^v}\right)=\sum_{n\geq 1}{\pi(g)j_n^o}+\sum_{n\geq 2}{\pi(g)j_n^v}.$$ For every $n\geq 1$ we have $\pi(g)j_n^o\in H_{n+1}^v$, and for every $n\geq 2$ we have $\pi(g)j_n^v\in H_{n-1}^o$. From the definition of the bilinear form $Q$ it is immediately seen that the norm of each $j_n^o$ and $j_n^v$ is preserved, and since all addends are orthogonal it follows that the norm of $j$ is preserved. \\

\begin{proof}[Proof of Theorem \ref{Main}, transitive case]
We have already reduced to real spherical representations, and more specifically to the explicit representations that correspond to $\alpha\in\mathbb{R}\setminus\{\pm 1\}$ in Theorem \ref{sph}. Due to Proposition \ref{qpres}, these representations preserve a bilinear form of index $0$ or $1$ on $H^{(\infty)}$, hence, due to Proposition \ref{uniqueform}, they cannot preserve any form of index $>1$. If a topologically irreducible continuous representation $(\pi,H)$ preserved a form of finite index $p>1$, the restriction to the space $H^{(\infty)}$ would also preserve a form of index $p$, a contradiction. \\
\end{proof}

\section{The non-transitive case}

In this section, we generalize the results of the previous two sections to the case where the action of $G$ on $V$ is not transitive. In particular, this will lead to a complete proof of Theorem \ref{Main}. \\

We recall that, due to Proposition \ref{props}(c), if the action of $G$ on $V$ is not transitive, then it has two orbits, one consisting of vertices $v$ with $d(o,v)$ even and one of those with $d(o,v)$ odd. \\

The tree must be regular or biregular. We may assume that it is biregular of degrees $r,s\geq 3$: the regular case will be the special case $r=s$. We may also assume, without loss of generality, that the base vertex $o$ has degree $r$. \\

We proceed as in Sects. 4 and 5, adapting the results to this new case. \\

Proposition \ref{subr} implies that $(\pi,H^{(\infty)})$ is equivalent to a subrepresentation of $(\Pi,C)$, where $C$ is the space of functions on $W:=\{v\in V: d(o,v)\text{ even}\}$ that are left-invariant under some compact open subgroup of $G$, and $\Pi$ is the left action of $G$ on this space. \\

We cannot use the standard Laplacian defined in Definition \ref{lapl}, so we introduce:

\begin{dfn}{(2-Laplacian)}
The \textit{2-Laplacian} on the tree $T$ is the operator $L_2: C\rightarrow C$ defined as follows: if $f: W\rightarrow\mathbb{K}$ is a function in $C$, the function $L_2 f$ associates to every vertex $v$ in $W$ the average of the values of $f$ at vertices $w$ satisfying $d(v,w)=2$. \\
\end{dfn}

Observation \ref{calclapl} changes in the following way: the 2-Laplacian of a radial function is again a radial function, with values
\begin{itemize}
    \item $L_2 f(0)=f(2)$,
    \item $L_2 f(2n)=\frac{1}{r(s-1)}f(2n-2)+\frac{s-2}{r(s-1)}f(2n)+\frac{r-1}{r}f(2n+2)$ for all $n\geq 2$. \\
\end{itemize}

Analogously to the transitive case, we can use the Gel'fand pair property of $(G,K)$ to show that there exists a radial function $f\in H^{(\infty)}$ with $f(0)=1$ such that the space of $K$-invariant functions in $H^{(\infty)}$ is the space of scalar multiples of $f$, and there exists $\alpha\in\mathbb{K}$ such that $L_2 f=\alpha\cdot f$. \\

We can see that there is one special case, namely $\alpha=1$, where the function $f$ is constant on $W$ and therefore $(\pi,H)$ is the trivial representation. In all other cases, we have $f(2)\neq f(0)$. We wish to deduce an analog of Theorem \ref{sph}. It turns out that the natural analog is always true, except in one particular case. \\

\begin{thm} \label{sph2}
If $\alpha\neq 1$ and $\alpha\neq -\frac{1}{s-1}$, then the natural analog of Theorem \ref{sph} holds: the representation $(\pi,H^{(\infty)})$ is equivalent to the left action of $G$ on the space $H_{\alpha}:=\{h\in C: L_2 h=\alpha\cdot h\}$, which is spanned by the $G$-translates of $f$, and conversely for every $\alpha\in\mathbb{K}$, $\alpha\neq 1$, $\alpha\neq -\frac{1}{s-1}$, the natural left action of $G$ on $H_{\alpha}$ gives a spherical algebraically irreducible representation of $G$.
\end{thm}

\begin{proof}
We can repeat the proof of Theorem \ref{sph}. The only part that needs to be changed (and that will lead to the exception for $\alpha=-\frac{1}{s-1}$) is when we show that the space $H_{\alpha}=\{h\in C: L_2 h=\alpha\cdot h\}$ is generated by the $G$-translates of the radial function $f$. To show this in our case, we fix a function $h$ in the space and try to build it with scalar multiples of translates of $f$. \\

We start with $F=f\cdot h(o)$, which agrees with $h$ at $o$. Now we wish to make it agree with $h$ on $S_2$ as well. For every $w\in S_2$, let $g_w\in G$ send $o$ to $w$. We add to $F$ the function $$\sum_{v\in N(o)}{\sum_{w\in N(v)}{\pi(g_w)f\cdot(f(0)-f(2))^{-1}(h(w)-F(w))}}.$$ The sum of the coefficients is zero, which guarantees that $F$ is unchanged at $o$. For every $v\in N(o)$, let $R(v):=\sum_{w\in N(v)}{(h(w)-F(w))}$. The condition $L_2 h=\alpha\cdot h$, together with $L_2 F=\alpha\cdot F$, implies that $\sum_{v\in N(o)}{R(v)}=0$. \\

For every $v\in N(o)$ and for every $w\in N(v)$, the function at $w$ has been changed by $$f(0)(f(0)-f(2))^{-1}(h(w)-F(w))$$ due to the contribution from $\pi(g_w)f$, by $$\sum_{z\in N(v)\setminus\{w\}}{f(2)(f(0)-f(2))^{-1}(h(z)-F(z))}=f(2)(f(0)-f(2))^{-1}(R(v)-(h(w)-F(w)))$$ due to the contribution from $\pi(g_z)f$ with $z\in N(v)\setminus\{w\}$, and by $$\sum_{x\in N(o)\setminus\{v\}}{\sum_{y\in N(x)}{f(4)(f(0)-f(2))^{-1}(h(y)-F(y))}}=-f(4)(f(0)-f(2))^{-1}R(v)$$ due to the contribution from $\pi(g_y)f$ with $y\in N(x)$ where $x\in N(o)\setminus\{v\}$. \\

All this adds up to $h(w)-F(w)+(f(2)-f(4))(f(0)-f(2))^{-1}R(v)$. \\

To remove the extra term $T(v):=(f(2)-f(4))(f(0)-f(2))^{-1}R(v)$, we subtract from $F$ the function $$\sum_{v\in N(o)}{\sum_{w\in N(v)}{\pi(g_w)f\cdot(f(0)+(s-2)f(2)-(s-1)f(4))^{-1}T(v)}}.$$

We need to assume that $f(0)+(s-2)f(2)-(s-1)f(4)\neq 0$. At every $w$ the value will be changed by $$-f(0)(f(0)+(s-2)f(2)-(s-1)f(4))^{-1}T(v)$$ due to the contribution from $\pi(g_w)f$, by $$-(s-2)f(2)(f(0)+(s-2)f(2)-(s-1)f(4))^{-1}T(v)$$ due to the contribution from $\pi(g_z)f$ with $z\in N(v)\setminus\{w\}$, and by
\begin{align*}
&\sum_{x\in N(o)\setminus\{v\}}{-(s-1)f(4)(f(0)+(s-2)f(2)-(s-1)f(4))^{-1}T(x)}= \\[0.5em]
&=(s-1)f(4)(f(0)+(s-2)f(2)-(s-1)f(4))^{-1}T(v)
\end{align*} due to the contribution from $\pi(g_y)f$ with $y\in N(x)$ where $x\in N(o)\setminus\{v\}$. \\

All this adds up to $-T(v)$, so now the function $F$ agrees with $h$ on $S_2$ as well as at $o$. \\

Like in the proof of Theorem \ref{sph}, we can follow the same procedure to make the two functions also agree on $S_4$ (applying the procedure to all vertices in $S_2$ in place of $o$), then on $S_6$ and so on, getting them to agree everywhere in a finite number of steps since $h$ is invariant under an open compact subgroup of $G$. \\

Therefore, the theorem is proved in the non-transitive case whenever $f(0)+(s-2)f(2)-(s-1)f(4)\neq 0$. As $f(2)=\alpha$ and
\begin{align*}
f(4)&=\frac{r}
{r-1}\left(\left(\alpha-\frac{s-2}{r(s-1)}\right)f(2)-\frac{1}{r(s-1)}f(0)\right)= \\
&=\frac{r(s-1)\alpha^2-(s-2)\alpha-1}{(r-1)(s-1)},
\end{align*}
the condition is true if and only if $1+(s-2)\alpha-\frac{r(s-1)\alpha^2-(s-2)\alpha-1}{(r-1)}\neq 0$, which is equivalent to $(s-1)\alpha^2-(s-2)\alpha-1\neq 0$, i.~e.~$((s-1)\alpha+1)(\alpha-1)\neq 0$, which holds for all $\alpha\neq 1,-\frac{1}{s-1}$. \\
\end{proof}

We are left with the case $\alpha=-\frac{1}{s-1}$. In that case, the function $f$ satisfies $f(2n)=\alpha^n$ for all $n\in\mathbb{N}$. This is checked directly for $n=0,1$ and by induction for $n\geq 2$: we have
\begin{align*}
f(2n)&=\frac{r}
{r-1}\left(\left(\alpha-\frac{s-2}{r(s-1)}\right)f(2n-2)-\frac{1}{r(s-1)}f(2n-4)\right)= \\
&=\frac{r}{r-1}\left(1+\frac{s-2}{r}-\frac{s-1}{r}\right)\alpha^n=\alpha^n. \\
\end{align*}

\begin{thm}
If $\alpha=-\frac{1}{s-1}$, the space $H^{(\infty)}$ is the subspace of $\{h\in C: L_2 h=\alpha\cdot h\}$ consisting of functions with sum zero on the set of neighbors of any vertex in $V\setminus W$ (i.~e.~any vertex at odd distance from $o$).
\end{thm}

\begin{proof}
As $f(2n)=\alpha^n$ for all $n$, it is easy to see that $f$ is in the space described in the statement: if a vertex is at distance $2k-1$ from $o$, the function will have value $\alpha^{k-1}$ at one neighbor and $\alpha^k$ at all other neighbors, adding up to $\alpha^{k-1}(1+(s-1)\alpha)=0$. It is also easy to see that the space is preserved under the action of $G$. \\

We only need to show that the $G$-translates of $f$ span our space. We can follow the exact same proof of Theorem \ref{sph2}, up to the point where we are left with the term $T(v)=(f(2)-f(4))(f(0)-f(2))^{-1}R(v)$. We cannot eliminate this term with the procedure that follows in the proof, because it would involve multiplying functions by the inverse of zero. However, we can show that in our case there is nothing to eliminate, as we have $R(v)=0$. \\

Recall that, in fact, $R(v)$ was the sum of $h(w)-F(w)$ where $w$ varies in the set $N(v)$. If we let $\overline{v}$ be the unique neighbor of $v$ which is not in $N(v)$ (i.~e.~the one on the path to $o$), since the sum of $h-F$ over $N(v)\cup\{\overline{v}\}$ is zero we have $R(v)=F(\overline{v})-h(\overline{v})$. The functions $F$ and $h$ already agree at $\overline{v}$, therefore $R(v)=0$. \\
\end{proof}

We can proceed to adapting the proof of Theorem \ref{Main} to the non-transitive case. We first assume $\alpha\neq -\frac{1}{s-1}$; the remaining case will be treated in the end. \\

We split $H^{(\infty)}=H_{\alpha}=\{h\in C: L_2 h=\alpha\cdot h\}$ into a direct sum of $K$-invariant subspaces: $H^{(\infty)}=\bigoplus_{n\in\mathbb{N}}{H_{n}}$. Again, $H_0$ is $1$-dimensional (spanned by a radial $f$). Then, for $n\geq 1$ we define $\widetilde{H}_n$ as the subspace of $H^{(\infty)}$ consisting of functions that are invariant under $G(B_n(o))$. In this way, a function in $\widetilde{H}_n$ will be determined by its values in $B_{2k}(o)$ if $n\in\{2k-1,2k\}$, and in the case $n=2k-1$ the value will have to be constant on $N(w)$ for every $w\in S_{2k-1}$. We can define, for every $k\geq 1$:

$$H_{2k-1}:=\left\{h\in\widetilde{H}_{2k-1}: h\vert_{B_{2k-2}(o)}\equiv 0,\ \sum_{w\in N(v)}{\sum_{x\in N(w)}{h(x)}}=0\ \forall v\in S_{2k-2}\right\};$$

$$H_{2k}:=\left\{h\in\widetilde{H}_{2k}: h\vert_{B_{2k-2}(o)}\equiv 0,\ \sum_{x\in N(w)}{h(x)}=0\ \forall w\in S_{2k-1}\right\}.$$ 

In this way, we have $\widetilde{H}_n=\bigoplus_{i=0}^{n}{H_i}$ for every $n$, and $H^{(\infty)}=\bigoplus_{n\in\mathbb{N}}{H_{n}}$. The dimensions of the subspaces $H_n$ are as follows:

\begin{itemize}
    \item $\dim(H_0)=1$,
    \item $\dim(H_1)=r-1$,
    \item $\dim(H_2)=r(s-2)$,
    \item $\dim(H_{2k-1})=r(r-2)(r-1)^{k-2}(s-1)^{k-1}$ for $k\geq 2$,
    \item $\dim(H_{2k})=r(r-1)^{k-1}(s-2)(s-1)^{k-1}$ for $k\geq 2$. \\
\end{itemize}

Like in the transitive case, we prove that $H_1$ is irreducible and that for every $n\geq 2$ each $K$-irreducible subspace of $H_n$ has dimension at least $\vert S_{n-1}\vert$. The proof is exactly the same if $n$ is even, while if $n=2k-1$ is odd, we consider, in place of functions on $S_{2k-1}$, functions on $S_{2k}$ that are constant on $N(w)$ for every $w\in S_{2k-1}$ (as is the case for functions in $H_{2k-1}$). \\

Since for every $k\geq 1$ we have $\vert S_{2k-1}\vert=r(r-1)^{k-1}(s-1)^{k-1}$ and $\vert S_{2k}\vert=r(r-1)^{k-1}(s-1)^k$, we always have $\vert S_{n-1}\vert>\dim{H_{n-1}}$. Hence, Corollary \ref{dec} holds, and therefore if the representation preserves a sesquilinear form $B$ then it must be the sum (for $n\in\mathbb{N}$) of $K$-invariant sesquilinear forms $B_n$ on $H_n$. \\

\begin{lem} \label{alphar2}
If the representation preserves the nondegenerate sesquilinear form $B$ (which we assume to be antilinear in the first variable and linear in the second), then $\alpha\in\mathbb{R}$.
\end{lem}

\begin{proof}
Let $f$ be the unique function in $H_0$ that satisfies $f(0)=1$. It satisfies $f(2)=\alpha$ and $f(4)=\frac{r(s-1)\alpha^2-(s-2)\alpha-1}{(r-1)(s-1)}$. Let $h$ be a function in $H_1\oplus H_2$ such that there exist $v\in N(o)$ and $x\in N(v)$ with $h(x)=1+\alpha$, $h(w)=\alpha$ for all $w\in N(v)\setminus\{x\}$, and $h(z)=-\frac{1+(s-1)\alpha}{(r-1)(s-1)}$ for all $z\in S_2\setminus N(v)$. \\

If we consider $g\in G$ that sends $o$ to $x$ and $x$ to $o$, we have $\pi(g)f=f\cdot\alpha+h\cdot(1-\alpha)$ and $\pi(g)h=f\cdot(1+\alpha)-h\cdot\alpha$. The former implies $$B(h,h)\vert 1-\alpha\vert^2=B(f,f)\left(1-\vert\alpha\vert^2\right).$$ Observe that $B(f,f)\neq 0$, because otherwise the form would be degenerate ($f$ is orthogonal to all functions in $\bigoplus_{n\geq 1}{H_n}$). \\

Since $B(f,h)=0$, it must be $B(\pi(g)f,\pi(g)h)=0$. This means that $B(f\cdot\alpha+h\cdot(1-\alpha),f\cdot(1+\alpha)-h\cdot\alpha)=0$, and therefore $$B(f,f)\overline{\alpha}(1+\alpha)-B(h,h)(1-\overline{\alpha})\alpha=0.$$ Multiplying by $1-\alpha$ on the left and using the equality above, recalling that $B(f,f)$ and $B(h,h)$ are in $\mathbb{R}$, we get $$B(f,f)(1-\alpha)\overline{\alpha}(1+\alpha)=B(h,h)\vert 1-\alpha\vert^2\alpha=B(f,f)\left(1-\vert\alpha\vert^2\right)\alpha.$$ We can divide by $B(f,f)$ since it is nonzero. Developing the two expressions, we get $$\overline{\alpha}-\vert\alpha\vert^2\alpha=\alpha-\vert\alpha\vert^2\alpha$$ and therefore $\overline{\alpha}=\alpha$, which means $\alpha\in\mathbb{R}$. \\
\end{proof}

Lemma \ref{alphar2} implies that Corollary \ref{tensor} also holds in the non-transitive case, and to finish the proof of Theorem \ref{Main} it suffices to show that no nondegenerate bilinear form of finite index $>1$ can be preserved by a spherical algebraically irreducible real representation of $G$. \\

Like in the transitive case, we finish the proof of Theorem \ref{Main} by presenting, for every $\alpha\in\mathbb{R}\setminus\{1\}$, an explicit nondegenerate bilinear form preserved under $G$. \\

For every $n\in\mathbb{N}$, a function in $H_n$ is uniquely determined by its values in $S_{2\lceil\frac{n}{2}\rceil}$. We define the bilinear form $Q_n$ on $H_n$ in such a way that it coincides with the restriction of the standard scalar product on the space of functions $S_{2\lceil\frac{n}{2}\rceil}\rightarrow\mathbb{R}$ to the subspace consisting of those that give rise to a function in $H_n$. \\

We define the bilinear form $Q$ on $H^{(\infty)}$ as follows: $$Q:=(1-\alpha)Q_0+\frac{r-1}{r(1+(s-1)\alpha)}\sum_{k\geq 1}{Q_{2k-1}}+\sum_{k\geq 1}{Q_{2k}}.$$

This bilinear form is nondegenerate with index $0$ if $-\frac{1}{s-1}<\alpha<1$, index $1$ if $\alpha>1$, and index $\infty$ if $\alpha<-\frac{1}{s-1}$. \\

We need to check that this bilinear form is preserved under $G$. Let $v\in N(o)$ and let $w\in N(v)$. Consider $g\in G$ which sends $o$ to $w$ and $w$ to $o$. The proof of Lemma \ref{gkgen} can be easily adapted to show that $g$ and $K$ generate $G$. As the bilinear form $Q$ is clearly preserved under $K$, it suffices to show that it is preserved under $g$. \\

Let $D:=\{o\}\cup N(v)$ be the set of neighbors of $v$. The space $H^{(\infty)}$ can be split into two subspaces: the space $H_D$ of functions that are invariant under $G(D)$, and the space $H_D^\perp$ of functions that are zero on $D$. We have $H^{(\infty)}=H_D\oplus H_D^\perp$: the intersection is easily seen to be trivial, and for every $h\in H^{(\infty)}$ there is a $G(D)$-invariant function that agrees with it on $D$ (built with translates of the radial function $f$ by elements of $G$ that fix $v$) and whose difference with $h$ will thus be zero on $D$. The notation $H_D^\perp$ is not casual, as it is easy to see that every function in $H_D$ is orthogonal to every function in $H_D^\perp$ (with respect to the bilinear form $Q$). It is also easy to see that the two subspaces are preserved under $g$ ($g$ preserves $D$), therefore it suffices to show that $g$ preserves $Q$ on $H_D$ and on $H_D^\perp$. \\

We first consider $H_D$. We split it into the direct sum $H_D^{(1)}\oplus H_D^{(2)}\oplus H_D^{(3)}$, where: $$H_D^{(1)}:=\left\{h\in H_D: h(o)=h(w),\ h\text{ constant on }D\setminus\{o,w\}\right\};$$ $$H_D^{(2)}:=\left\{j\in H_D: j(o)+j(w)=0,\ j(x)=0\ \forall x\in D\setminus\{o,w\}\right\};$$ $$H_D^{(3)}=\left\{l\in H_D: l(0)=l(w)=0,\ \sum_{x\in D\setminus\{o,w\}}{l(x)}=0\right\}.$$ \\

The three subspaces are preserved by $g$, as $g$ swaps $o$ and $w$. \\

Let $h,j,l$ be arbitrary functions in $H_D^{(1)},H_D^{(2)},H_D^{(3)}$ respectively. We have $\pi(g)h=h$, therefore $g$ preserves $Q$ on $H_D^{(1)}$. We also have $\pi(g)j=-j$, therefore $g$ preserves $Q$ on $H_D^{(2)}$. The functions $l$ and $\pi(g)l$ are both in $H_2$ with the same norm, therefore $g$ preserves $Q$ on $H_D^{(3)}$. \\

We claim that the spaces $H_D^{(1)}$, $H_D^{(2)}$, $H_D^{(3)}$ are mutually orthogonal with respect to the bilinear form $Q$, which will allow us to conclude that $g$ preserves $Q$ on $H_D$. For this, we need to show that $h,j,l$ are mutually orthogonal. \\

First, we see that $l$ is orthogonal to both $h$ and $j$, because its restriction to $S_2$ is supported (with sum zero) on $N(v)\setminus\{w\}$ while $h$ and $j$ are constant on this set.\\

Now, we need to show that $Q(h,j)=0$. We may assume (up to multiplying by a scalar) that $j(o)=1$, which implies $j(w)=-1$. Let $a$ be the value of $h$ at $o$ (or $w$) and let $b$ be its value at any $x\in D\setminus\{o,w\}$. The function $h$ is the sum of the following:

\begin{itemize}
    \item a function in $H_0$ with value $a$ at $o$,
    \item a function in $H_1$ with value $\left(\frac{1}{s-1}-\alpha\right)a+\frac{s-2}{s-1}b$ at any $x\in N(v)$ and value $-\frac{1}{r-1}\left(\left(\frac{1}{s-1}-\alpha\right)a+\frac{s-2}{s-1}b\right)$ at any $x\in S_2\setminus N(v)$,
    \item a function in $H_2$ supported on $N(v)$ with value $\frac{s-2}{s-1}(a-b)$ at $w$ and value $-\frac{1}{s-1}(a-b)$ at any $x\in N(v)\setminus\{w\}$. \\
\end{itemize}

The function $j$ is the sum of the following:

\begin{itemize}
    \item a function in $H_0$ with value $1$ at $o$,
    \item a function in $H_1$ with value $-\frac{1}{s-1}-\alpha$ at any $x\in N(v)$ and value $\frac{1}{r-1}\left(\frac{1}{s-1}+\alpha\right)$ at any $x\in S_2\setminus N(v)$,
    \item a function in $H_2$ supported on $N(v)$ with value $-\frac{s-2}{s-1}$ at $w$ and $\frac{1}{s-1}$ at any $x\in N(v)\setminus\{w\}$. \\
\end{itemize}

We have:

\begin{align*}
Q(h,j)&=(1-\alpha)a+\frac{r-1}{r(1+(s-1)\alpha)}(s-1)\left(\left(\frac{1}{s-1}-\alpha\right)a+\frac{s-2}{s-1}b\right)\left(-\frac{1}{s-1}-\alpha\right)+ \\
&+\frac{r-1}{r(1+(s-1)\alpha)}(r-1)(s-1)\left(-\frac{1}{r-1}\left(\left(\frac{1}{s-1}-\alpha\right)a+\frac{s-2}{s-1}b\right)\right)\frac{1}{r-1}\left(\frac{1}{s-1}+\alpha\right)+ \\
&+\frac{s-2}{s-1}(a-b)\left(-\frac{s-2}{s-1}\right)+(s-2)\left(-\frac{1}{s-1}(a-b)\right)\frac{1}{s-1}= \\[1em]
&=(1-\alpha)a-\frac{r-1}{r(1+(s-1)\alpha)}\left(\left(1-(s-1)\alpha\right)a+(s-2)b\right)\left(\frac{1}{s-1}+\alpha\right)+ \\
&-\frac{1}{r(1+(s-1)\alpha)}\left(\left(1-(s-1)\alpha\right)a+(s-2)b\right)\left(\frac{1}{s-1}+\alpha\right)-\frac{(s-2)^2+(s-2)}{(s-1)^2}(a-b)= \\[1em]
&=(1-\alpha)a-\frac{1}{s-1}\left(\left(1-(s-1)\alpha\right)a+(s-2)b\right)-\frac{s-2}{s-1}(a-b)= \\[1em]
&=\frac{1}{s-1}(((1-\alpha)(s-1)-(1-(s-1)\alpha)-(s-2))a-((s-2)-(s-2))b)=0. \\
\end{align*}

Therefore, we have shown that $g$ preserves $Q$ on $H_D$. We now show that $g$ preserves $Q$ on $H_D^\perp$, in a way similar to the end of the proof of Proposition \ref{qpres}. \\

We split each $H_n$ ($n\geq 3$) in three orthogonal subspaces: $H_n=H_n^o\oplus H_n^x\oplus H_n^w$, where $H_n^o$ is the space of functions in $H_n$ supported on $\{z\in W: d(o,z)<d(w,z)\}$ (which means that the path from $z$ to $v$ meets $D$ at $o$), $H_n^x$ is the space of functions in $H_n$ supported on $\{z\in W: d(o,z)=d(w,z)\}$ (which means that the path from $z$ to $v$ meets $D$ at a vertex other than $o$ and $w$), and $H_n^w$ is the space of functions in $H_n$ supported on $\{z\in W: d(o,z)>d(w,z)\}$ (which means that the path from $z$ to $v$ meets $D$ at $w$). We can also define $H_1^o$ as the space of functions in $H_1$ supported on $\{z\in W: d(o,z)<d(v,z)\}$, and $H_2^o$ in the same way. Observe that a function in $H_D^\perp$ lies in $H_1^o\oplus H_2^o\oplus\bigoplus_{n\geq 3}{H_n}$. \\

For every $q\in H_D^\perp$, we call $q_n^o$ its projection to $H_n^o$, $q_n^x$ its projection to $H_n^x$, and $q_n^w$ its projection to $H_n^w$. We have $$\pi(g)q=\pi(g)\left(\sum_{n\geq 1}{q_n^o}+\sum_{n\geq 3}{q_n^x}+\sum_{n\geq 3}{q_n^w}\right)=\sum_{n\geq 1}{\pi(g)q_n^o}+\sum_{n\geq 3}{\pi(g)q_n^x}+\sum_{n\geq 3}{\pi(g)q_n^w}.$$ For every $n\geq 1$ we have $\pi(g)q_n^o\in H_{n+2}^w$, for every $n\geq 3$ we have $\pi(g)q_n^x\in H_n^x$, and for every $n\geq 3$ we have $\pi(g)q_n^w\in H_{n-2}^o$. From the definition of the bilinear form $Q$ it is then immediate to see that the norm of each $q_n^o$, $q_n^x$ and $q_n^w$ is preserved, and as all the addends are orthogonal, it follows that the norm of $q$ is preserved.

We are only left with the case $\alpha=-\frac{1}{s-1}$. \\

We claim that, in this case, all spaces $H_{2k-1}$ ($k\geq 1$) are trivial. Recall that if $\alpha=-\frac{1}{s-1}$, every function in $H^{(\infty)}$ has sum zero on the set of neighbors of any vertex at odd distance from $o$. Recall also that a function in $H_{2k-1}$ is zero on $B_{2k-2}(o)$ and constant on $N(v)$ for every $v\in S_{2k-1}$. As the sum over neighbors of $v$ has to be zero, the sum over $N(v)$ has to be zero, which means that the function is identically zero on $N(v)$. This shows that the function is identically zero on $S_{2k}$. Then, since the function is invariant under $G(B_{2k-1}(o))$, it has to be zero everywhere. \\

For every even $n$ we can define the bilinear form $Q_n$ on $H_n$ as in the case $\alpha\neq -\frac{1}{s-1}$, then we can define the bilinear form $Q$ on $H^{(\infty)}$ as follows: $$Q:=(1-\alpha)Q_0+\sum_{k\geq 1}{Q_{2k}}.$$ \\

This form has index $0$ (it is positive definite). To verify that it is preserved under $G$, the proof is the same as in the case $\alpha\neq -\frac{1}{s-1}$, only the verification that $g$ preserves $Q$ on $H_D$ must be adapted. \\

We can split $H_D$ into the same direct sum $H_D^{(1)}\oplus H_D^{(2)}\oplus H_D^{(3)}$. The only difference is that, while in the case $\alpha\neq -\frac{1}{s-1}$ the space $H_D^{(1)}$ was $2$-dimensional (it was the space of functions in $H_D$ with a certain value $a$ at $o$ and $w$ and a certain value $b$ at every $x\in D\setminus\{o,w\}$), in our case it is only $1$-dimensional: the two values $a$ and $b$ have to satisfy $2a+(s-2)b=0$. \\

The proof that $g$ preserves the three subspaces, that it preserves $Q$ on each of them, and that $H_D^{(3)}$ is orthogonal to both of the others works exactly as in the case $\alpha\neq -\frac{1}{s-1}$. The only part to change is where we check that $H_D^{(1)}$ and $H_D^{(2)}$ are orthogonal. \\

We use the same notation as before and fix arbitrary $h\in H_D^{(1)}$, $j\in H_D^{(2)}$. We let $a$ be the value of $h$ at $o$ and $w$, $b$ be the value of $h$ at any $x\in D\setminus\{o,w\}$ (they must satisfy $2a+(s-2)b=0$), and we assume without loss of generality $j(o)=1$ and $j(w)=-1$. We have

\begin{align*}
Q(h,j)&=(1-\alpha)a+\frac{s-2}{s-1}(a-b)\left(-\frac{s-2}{s-1}\right)+(s-2)\left(-\frac{1}{s-1}(a-b)\right)\frac{1}{s-1}= \\[1em]
&=\left(1+\frac{1}{s-1}\right)a-\frac{s-2}{s-1}(a-b)=\frac{1}{s-1}(2a+(s-2)b)=0. \\
\end{align*}

This concludes the proof for the case $\alpha=-\frac{1}{s-1}$. \\

\begin{proof}[Proof of Theorem \ref{Main}, non-transitive case]
We have reduced to real spherical representations corresponding to $\alpha\in\mathbb{R}\setminus\{1,-\frac{1}{s-1}\}$ in Theorem \ref{sph2}, together with the extra representation arising for $\alpha=-\frac{1}{s-1}$. We have shown that, in all these cases, the representation preserves a bilinear form of index $0$, $1$ or $\infty$ on $H^{(\infty)}$. In the case of index $0$ or $1$, Proposition \ref{uniqueform} applies directly and excludes that a form of finite index $p>1$ can be preserved, while in the case of index $\infty$ we can argue by contradiction that if $G$ preserved a form of finite index $p$ it could not preserve (again by Proposition \ref{uniqueform}) a form of index $\infty$ on $H^{(\infty)}$. Like in the transitive case, if the algebraic representation to $H^{(\infty)}$ cannot preserve a form of finite index $p>1$, the topological representation to $H$ cannot either. \\
\end{proof}

\section*{Acknowledgments}

This paper is part of my Ph.D. project under the supervision of Prof. Nicolas Monod at EPFL (Lausanne, Switzerland). I am grateful to Nicolas for proposing this project, for giving me valuable advice and feedback, and for sharing with me deep insights of the theory. I am also grateful to Bruno Duchesne for his useful questions and remarks which helped me make this paper more clear, and to the referee for several comments that helped me improve the exposition and correct a few small mistakes. \\

\section*{Statement on AI Use}

The research presented in this paper was carried out almost entirely without any use of generative artificial intelligence tools. Only in the later stages of preparation and revision, some basic tools
such as the free versions of Gemini and ChatGPT have been used for literature search and a few minor technical steps. In particular, Gemini 3.1 Pro has given the author some useful ideas for the proof of Proposition \ref{subirredC} and helped show that admissibility of complex representations implies admissibility of real and quaternionic representations, and ChatGPT has found the reference \cite{quantum} for the quaternionic Schur Lemma. The text of the paper has been entirely written by the human author, who takes full responsibility for its contents. \\

\appendix
\section{Appendix}

In this appendix, we prove that, in the situation of Theorem \ref{Main}, namely in presence of a preserved strongly nondegenerate sesquilinear form of finite index, the irreducible representation $(\pi,H)$ is admissible, and that this implies that its restriction to the subspace $H^{(\infty)}$ of smooth vectors is irreducible in the algebraic sense, as we claimed at the end of Section 2. \\

First, we recall the classical definition of admissibility, together with a well-known equivalence:

\begin{dfn}
    Let $K$ be a maximal compact subgroup of $G$. The representation $(\pi,H)$ is \textit{admissible} if every irreducible representation of $K$ appears with finite multiplicity in the restriction of $(\pi,H)$ to $K$. This is equivalent to the property that $\dim{H^U}<\infty$ for every open compact $U<G$. \\
\end{dfn}

We start with a continuous topologically irreducible $\mathbb{K}$-representation $(\pi,H)$, with $\mathbb{K}\in\{\mathbb{R},\mathbb{C},\mathbb{H}\}$. We can assume $H$ to be infinite-dimensional; otherwise, what we want to prove is trivial. We consider the algebra $C_c^\infty(G)$, defined as the algebra of continuous compactly supported functions on $G$ that are locally constant, i.~e.~invariant under some open compact subgroup of $G$. We consider complex functions if $\mathbb{K}=\mathbb{C}$ and real functions in the other two cases. The algebra $C_c^\infty(G)$ acts irreducibly on $H$, meaning it does not preserve any non-trivial closed $\mathbb{K}$-subspace. \\

\begin{lem} \label{irred}
    Let $U<G$ be an open compact subgroup, and consider the algebra $E(U)\subseteq C_c^\infty(G)$ consisting of functions that are left- and right-invariant with respect to $U$. This algebra acts irreducibly on $H^U$.
\end{lem}

\begin{proof}
    First, observe that $E(U)=\chi_U * C_c^\infty(G) * \chi_U$, where $\chi_U$ denotes the indicator function of $U$ and $*$ denotes the convolution. The operator $\chi_U$ is, up to a constant (the Haar measure of $U$), the averaging operator over $U$. Therefore, it is clear that every operator in $E(U)$ has image contained in $H^U$. Moreover, if $W\subseteq H^U$ is $E(U)$-invariant and nonzero, then by irreducibility $C_c^\infty(G)W$ is dense in $H$, therefore $\chi_U C_c^\infty(G)W=\chi_U C_c^\infty(G)\chi_U W=E(U)W\subseteq W$ is dense in $H^U$. \\
\end{proof}

\begin{prp} \label{algirred}
    If the representation $(\pi,H)$ is admissible, then the restriction of the representation to $H^{(\infty)}$ is irreducible in the algebraic sense.
\end{prp}

\begin{proof}
    If $(\pi,H)$ is admissible, $H^U$ is finite-dimensional for every open compact $U<G$. In this case, if an invariant subspace $W\subseteq H^U$ is dense it must be $W=H^U$. Therefore, Lemma \ref{irred} means that the algebra $E(U)$ acts on $H^U$ with no invariant subspace except $\{0\}$ and $H^U$. Now, for every $v_1,v_2\in H^{(\infty)}$ there exist open compact $U_1,U_2<G$ such that $v_1\in H^{U_1}$ and $v_2\in H^{U_2}$; if we consider the open compact $U=U_1\cap U_2$ we have $v_1,v_2\in H^U$, therefore, if we denote by $v_1\mathbb{K}$ the $\mathbb{K}$-span of $v_1$, we have $v_2\in E(U) v_1\mathbb{K}\subseteq C_c^\infty(G) v_1\mathbb{K}$. See also (\cite{renard}, III.1.5) and (\cite{carterwhite}, Proposition 2.12). \\
\end{proof}

Due to Proposition \ref{algirred}, we are only left to prove admissibility of topologically irreducible representations. First, we will restrict our attention to the case $\mathbb{K}=\mathbb{C}$, then we will show how admissibility in this case implies admissibility in the other cases. \\

\begin{rmk}
    It is already known that all irreducible unitary $\mathbb{C}$-representations of $G$ are admissible. This was already implicit in \cite{olsh2}, and is further explained at the end of Chapter III.3 in \cite{talanebbia} in the case $G=\mathrm{Aut}(T)$ for a regular tree $T$, following the classification. In our generality, unitary representations are classified in \cite{amann}, and one can apply the same argument to deduce admissibility. \\
\end{rmk}

We show a general result about existence of admissible representations which, following the above remark, can be applied in our case.

\begin{prp} \label{exadm}
    Let $G$ be a locally compact group. Let $C_c(G)$ be the algebra of continuous compactly supported complex functions on $G$. Assume that all continuous irreducible unitary $\mathbb{C}$-representations of $G$ are admissible. Then for every $h\in C_c(G)$ there exists a continuous irreducible admissible unitary representation $\phi$ of $G$ such that $\phi(h)\neq 0$.
\end{prp}

\begin{proof}
     Since $h\in C_c(G)$, we have $h\in L^1(G)$ and therefore $h\in\mathbb{C}^*(G)$, where $\mathbb{C}^*(G)$ is the $\mathbb{C}^*$-algebra of $G$ as defined in (13.9) of \cite{dixmier}. By (\cite{kadrin}, Corollary 10.2.4) there exists an irreducible representation $\Phi$ of $\mathbb{C}^*(G)$ such that $\Phi(h)\neq 0$. This representation comes from the Gelfand-Naimark-Sagel (GNS) construction and is nondegenerate. By (\cite{dixmier}, 13.9.3) there exists a bijective correspondence between nondegenerate representations of $\mathbb{C}^*(G)$ and continuous unitary representations of $G$. Moreover, this correspondence preserves irreducibility (\cite{dixmier}, 13.3.5). This means that the representation $\Phi$ is associated with a continuous irreducible unitary representation of $G$; then our hypothesis implies that this representation is also admissible. \\
\end{proof}

We turn back to our algebra $C_c^{\infty}(G)$, which acts irreducibly on $H$. We know that the representation $(\pi,H)$ preserves a continuous strongly nondegenerate sesquilinear form on $H$, and we can associate to this sesquilinear form an adjoint operation $\dag$ on $H$. The algebra $\pi(C_c^\infty(G))$ is closed under the $\dag$-operation, because if $f\in C_c^{\infty}(G)$ we can define $\hat{f}(g)=\overline{f(g^{-1})}$ and we have $\pi(\hat{f})=\pi(f)^\dag$, like in Lemma \ref{symmetric}. \\

We can now show a general result that allows to deduce that $\pi(C_c^\infty(G))$ is dense in the space of bounded linear operators on $H$. \\

\begin{thm} \label{dense}
    Let $H$ be a complex infinite-dimensional Hilbert space equipped with a continuous strongly nondegenerate sesquilinear form of finite index, which generates an adjoint operation $\dag$. Let $\mathcal{A}$ be an algebra of bounded linear operators on $H$. Suppose that $\mathcal{A}$ is closed under the $\dag$-operation and that it acts irreducibly on $H$, meaning that there are no nontrivial $\mathcal{A}$-invariant closed subspaces. Then $\mathcal{A}$ is dense in the space $L(H)$ of bounded linear operators on $H$, with respect to the strong operator topology. \\
\end{thm}

We first observe that, if the sesquilinear form were positive definite, then this would be a well-known classical result, following from the Schur Lemma and the Von Neumann double commutant theorem (\cite{takesaki}, Chapter II, Theorem 3.9). We will prove Theorem \ref{dense} by first showing that the Schur Lemma can also be applied in this case and then adapting a special case of the double commutant theorem. \\

\begin{prp}
    In the setting of Theorem \ref{dense}, let $\mathcal{A}'\subset L(H)$ be the commutator of $\mathcal{A}$, i.~e.~the algebra of operators that commute with all operators in $\mathcal{A}$. Then $\mathcal{A}'$ only consists of scalar multiples of the identity.
\end{prp}

\begin{proof}
     The algebra $\mathcal{A}'$ is closed under the $\dag$-operation, because $h\in L(H)$ commutes with $f\in\mathcal{A}$ if and only if $h^\dag$ commutes with $f^\dag$, and $\mathcal{A}$ is closed under the $\dag$-operation. \\
    
    Let $h\in\mathcal{A}'$. We want to show that $h$ is a scalar multiple of the identity. If $h$ is self-adjoint for the $\dag$-operation, then by Theorem 12.1' in \cite{iohkrlan} there exists a finite-dimensional nonnegative subspace of $H$ that is preserved by $h$, which implies that $h$ has at least one nonnegative eigenvector $v\neq 0$. The eigenspace to which it belongs is closed and $\mathcal{A}$-invariant, because if $hv=\lambda v$ for $\lambda\in\mathbb{C}$ then for every $f\in\mathcal{A}$ we have $hfv=fhv=f(\lambda v)=\lambda fv$. Therefore, since $\mathcal{A}$ acts irreducibly on $H$, the eigenspace coincides with $H$ and $h=\lambda Id$. \\

    For a general $h$ we can write $h=\frac{h+h^\dag}{2}-i\frac{ih+(ih)^\dag}{2}$. The operators $\frac{h+h^\dag}{2}$ and $\frac{ih+(ih)^\dag}{2}$ are in $\mathcal{A}'$ and self-adjoint for the $\dag$-operation, hence scalar multiples of the identity by the above argument. This completes the proof. \\
\end{proof}

\begin{prp}
    In the setting of Theorem \ref{dense}, the algebra $\mathcal{A}$ is dense in its double commutator $\mathcal{A}''$ (defined as the commutator of the commutator of $\mathcal{A}$) in the strong operator topology. \\
\end{prp}

\begin{proof}
    Let $h\in\mathcal{A}''$. We have to show that $h$ can be approximated by elements of $\mathcal{A}$ in the strong operator topology. This means that for every finite sequence $(v_1,\ldots,v_n)$ of vectors in $H$ and for every $\epsilon>0$ there exists $f\in\mathcal{A}$ such that $||hv_i-fv_i||<\epsilon$ for $1\leq i\leq n$. \\

    For every algebra $\mathcal{B}\subseteq L(H)$, we call $\mathcal{B}^n\subseteq L(H^n)$ the algebra defined by the diagonal embedding of $\mathcal{B}$, meaning that a general element of $\mathcal{B}^n$ sends every $(v_1,\ldots,v_n)$ to $(bv_1,\ldots,bv_n)$ for some $b\in\mathcal{B}$. \\
    
    We have $(\mathcal{A}^n)''=(\mathcal{A}'')^n$ (\cite{takesaki}, Chapter II, Corollary 3.6). Therefore, for $h\in\mathcal{A}''$, the operator $\widetilde{h}: (v_1,\ldots,v_n)\mapsto (hv_1,\ldots,hv_n)$ is in $(\mathcal{A}^n)''$. \\

    The Hilbert space $H^n$ is naturally equipped with a strongly nondegenerate sesquilinear form coming from $H$, whose index is $n$ times the index of the form on $H$. It is immediate to see that $\mathcal{A}^n$ is closed under taking adjoints with respect to this form (we will denote again by $\dag$ this operation). \\
    
     Let $\overline{v}=(v_1,\ldots,v_n)$ be fixed. Let $W\subseteq H^n$ be the closure of $\mathcal{A}^n\overline{v}$. We claim that $W$ is nondegenerate for the sesquilinear form on $H^n$. If that were not the case, we would have $W\cap W^\perp\neq\{0\}$, where $W^\perp$ denotes the orthogonal of $W$ with respect to the sesquilinear form. The restriction of the form to $W\cap W^\perp$ is identically zero; therefore, the dimension of $W\cap W^\perp$ is at most the index of the sesquilinear form, which is finite. Since $W$ is preserved by $\mathcal{A}^n$, so is $W^\perp$: for $a\in\mathcal{A}^n$, $z\in W^\perp$ we have $\langle az,w\rangle=\langle z,a^\dag w\rangle=0$ for all $w\in W$, since $A^n$ is closed under the $\dag$-operation. Therefore, $W\cap W^\perp$ is preserved by $\mathcal{A}^n$, which implies that its projection to the first coordinate of the product $H\times\cdots\times H=H^n$ is preserved by $\mathcal{A}$. This is a finite-dimensional subspace of $H$ preserved by $\mathcal{A}$, which contradicts the fact that $\mathcal{A}$ acts irreducibly on $H$. \\
     
     We have shown that $W$ is nondegenerate, which means that there exists an orthogonal projection $p: H^n\rightarrow W$. This projection commutes with every $a\in\mathcal{A}^n$: for all $v\in H^n$ we can write $v=w+w^\perp$ with $w\in W$, $w^\perp\in W^\perp$, and we have $pav=paw+paw^\perp=aw=apv$, where we have used that both $W$ and $W^\perp$ are preserved by $a\in\mathcal{A}^n$. Therefore, we have $p\in(\mathcal{A}^n)'$, which means that $p$ commutes with every $\alpha\in(\mathcal{A}^n)''$. This means that every such $\alpha$ preserves $W$: for all $w\in W$ we have $\alpha w=\alpha pw=p\alpha w$ so $\alpha w\in W$. Choosing $\alpha$ to be the operator $\widetilde{h}$ sending $\overline{v}=(v_1,\ldots,v_n)$ to $\widetilde{h}\overline{v}=(hv_1,\ldots,hv_n)$, we find that $\widetilde{h}\overline{v}$ is in the closure of $\mathcal{A}^n\overline{v}$, which means that for every $\epsilon>0$ there exists $f\in\mathcal{A}$ such that $||hv_i-fv_i||<\epsilon$ for $1\leq i\leq n$, as we wanted to prove. \\
\end{proof}

We have completed the proof of Theorem \ref{dense}. We are now ready to prove the admissibility of our representations. We will put together arguments from the first proof of admissibility of irreducible algebraic representations of the automorphism group of a regular tree, by Ol'shanskii \cite{olsh2}, and from the proof of admissibility of irreducible unitary representations of linear connected reductive groups (\cite{knapp2}, Theorem 8.1). \\

\begin{thm} \label{admissC}
    Let $G<\mathrm{Aut}(T)$ be as in Theorem \ref{Main}, and let $(\pi,V)$ be a continuous irreducible representation of $G$ in the complex Hilbert space $V$, preserving a continuous strongly nondegenerate sesquilinear form of finite index. Then $(\pi,V)$ is admissible.
\end{thm}

\begin{proof}
    It is enough to show $\dim{H^U}<\infty$ when $U$ belongs to a fixed basis of compact open neighborhoods of the identity in $G$. We can choose the basis of neighborhoods $(U_n)_{n\in\mathbb{N}}$, where $U_n$ is the pointwise stabilizer of the set of vertices of $T$ at distance $\leq n$ from a fixed base vertex $o$. \\

    For every $n\in\mathbb{N}$, the algebra $E(U_n)=\chi_{U_n} * C_c^\infty(G) * \chi_{U_n}$ acts irreducibly on $H^{U_n}$ (Lemma \ref{irred}). \\
    
    As explained in Lemma 2 of \cite{olsh2}, we have $E(U_n)=E_n * F_n * E_n$, where $E_n=\{f\in E(U_n): supp(f)\subseteq U_0\}$ and $F_n=\{f\in E(U_n): supp(f)\subseteq\bigcup_{k\geq0}{U_n t^k U_n}\}$, with $t$ denoting an element of $G$ that moves by one all vertices in a geodesic passing through the base vertex $o$ (or by two if the action of $G$ is not transitive on the vertices, in which case we only consider even values of $n$). If $n\geq 1$, $F_n$ is a subalgebra generated by a single element: the function with support $U_n t U_n$. On the other hand, $E_n$ has finite dimension $m=|U_0/U_n|$. \\

    As discussed in the proof of Theorem 8.1 in \cite{knapp2}, for every positive integer $d$ there exists a minimal $r(d)$ such that
    \begin{equation} \label{msi}
        \sum_{\varepsilon\in S_{r(d)}}{(\mathrm{sgn}\ \varepsilon)X_{\varepsilon(1)}X_{\varepsilon(2)}\cdots X_{\varepsilon(r(d))}}=0 \ \ \ \forall X_1,\ldots,X_{r(d)}\in M_d(\mathbb{C}),
    \end{equation}
    and this minimal $r(d)$ is a strictly increasing function of $d$. We call the identity \eqref{msi} the \textit{minimal standard identity} for a matrix algebra of degree $d$. \\
    
    We claim that the algebra $E(U_n)$ satisfies the minimal standard identity for a matrix algebra of degree $m^2$. If that were not the case, then there would exist $X_1,\ldots,X_{r(m^2)}\in E(U_n)$ such that the expression in \eqref{msi} is nonzero. We call $X$ the result of this expression, which is an operator in $E(U_n)$. Consider now a continuous irreducible admissible unitary representation $(\phi,S)$ of $G$ on a Hilbert space $S$ that satisfies $\phi(X)\neq 0$; such a representation exists by Proposition \ref{exadm}. By Lemma \ref{irred}, $E(U_n)$ acts irreducibly on $S^{U_n}$. Let $r=\dim{S^{U_n}}$; if we identify $S^{U_n}$ with $\mathbb{C}^r$ we have $\phi(E(U_n))=\mathrm{End}(\mathbb{C}^r)$ because irreducibility implies density of the algebra in $L(\mathbb{C}^r)=\mathrm{End}(\mathbb{C}^r)$ (using Theorem \ref{dense} for unitary representations, which is a classical result). We know that $\phi(E(U_n))=\phi(E_n)\phi(F_n)\phi(E_n)$, where $\phi(E_n)$ has dimension at most $m$ and $\phi(F_n)$ is a subalgebra of $\mathrm{End}(\mathbb{C}^r)$ generated by a single element and therefore has dimension $\leq r$. Therefore, we have $m^2 r\geq\dim{\phi(E(U_n))}=r^2$, and hence $r\leq m^2$. This means that $\mathrm{End}(\mathbb{C}^r)$ satisfies the minimal standard identity for a matrix algebra of degree $m^2$, which contradicts $\phi(X)\neq 0$. \\

    Now we turn back to our original representation $(\pi,H)$ and claim $\dim{H^{U_n}}\leq m^2$. We know that the algebra $E(U_n)$ acts irreducibly on $H^{U_n}$ and satisfies the minimal standard identity for a matrix algebra of degree $m^2$. By Theorem \ref{dense} the algebra $C_c^\infty(G)$ is dense in $L(H)$ in the strong operator topology. This implies that the algebra $E(U_n)=\chi_{U_n} * C_c^\infty(G) * \chi_{U_n}$ is dense in $\chi_{U_n} * L(H) * \chi_{U_n}=L(H^{U_n})$. Therefore, the minimal standard identity holds for the algebra $L(H^{U_n})$, which implies $\dim{H^{U_n}}\leq m^2$. \\
\end{proof}

Theorem \ref{admissC} provides admissibility of representations in the case $\mathbb{K}=\mathbb{C}$. We are left to show that this also implies admissibility in the cases $\mathbb{K}=\mathbb{R}$ and $\mathbb{K}=\mathbb{H}$. A fundamental step to this end is the following:

\begin{prp} \label{subirredC}
    Let $(\pi,H)$ be a continuous topologically irreducible $\mathbb{R}$-representation of $G$, preserving a continuous strongly nondegenerate indefinite bilinear form of finite index on the real Hilbert space $H$. Then its complexification $H_{\mathbb{C}}$ (defined as the tensor product $H\otimes_{\mathbb{R}}\mathbb{C}$) contains a closed $G$-invariant $\mathbb{C}$-subspace which is a topologically irreducible $\mathbb{C}$-subrepresentation of $H_{\mathbb{C}}$.
\end{prp}

\begin{proof}
    We assume $H$ to be infinite-dimensional; otherwise the result is trivial. The complexification $H_\mathbb{C}$, when viewed as a real vector space, is isomorphic to the direct sum $H\oplus H$, with $G$ acting by $\pi$ on each summand. If $W\subseteq H_\mathbb{C}$ is a nonzero closed $G$-invariant $\mathbb{C}$-subspace, it is also a closed $G$-invariant $\mathbb{R}$-subspace of $H\oplus H$. As $H$ is irreducible, the projection of $W$ to each summand must be trivial or dense. This implies that if $W\neq\{0\}$, then $W$ must be infinite-dimensional. Moreover, since $G$ preserves the bilinear form on $H$ (which naturally extends to a sesquilinear form on $H_\mathbb{C}$), the orthogonal space $W^\perp$ is also a closed $G$-invariant $\mathbb{C}$-subspace of $H_{\mathbb{C}}$. Then so is its intersection $W\cap W^\perp$, which is finite-dimensional since the form has finite index. Therefore, $W\cap W^\perp=\{0\}$, which means that $W$ must be nondegenerate (i.~e.~the sesquilinear form restricted to $W$ is nondegenerate). \\

    Let $\mathcal{W}$ be the family of all nonzero closed $G$-invariant $\mathbb{C}$-subspaces of $H_\mathbb{C}$ where the sesquilinear form is nondegenerate and indefinite, ordered by inclusion. Clearly $\mathcal{W}$ is not empty, as it contains $H_\mathbb{C}$. By the Hausdorff maximal principle, this partially ordered set has a maximal totally ordered chain, which we may call $\mathcal{C}$. We claim that the intersection $W_{\infty}$ of all $W\in\mathcal{C}$ is in $\mathcal{W}$, which would imply, by maximality of $\mathcal{C}$, that $W_{\infty}\in\mathcal{C}$ and $W_{\infty}$ is minimal in $\mathcal{W}$. \\
    
    Clearly, $W_{\infty}$ is closed and $G$-invariant, and if $W_{\infty}\neq\{0\}$ the restriction of the sesquilinear form to $W_{\infty}$ must be nondegenerate by the argument above. Each $W\in\mathcal{C}$ has a finite index $k(W)$, which is the dimension of a maximal negative definite subspace in $W$ (while the maximal dimension of a positive definite subspace is infinite). Clearly, $W\subseteq W'$ implies $k(W)\leq k(W')$. Therefore, if we let $k_0\geq 1$ be the minimal index of any $W\in\mathcal{C}$, we might fix a $W_0\in\mathcal{C}$ with $k(W_0)=k_0$ and restrict the chain $\mathcal{C}$ to $\mathcal{C}_0=\{W\in\mathcal{C}: W\subseteq W_0\}$; these subspaces have all index $k_0$ and their intersection is still $W_{\infty}$. For every $W\in\mathcal{C}_0$, the orthogonal complement $W^{(\perp)}$ of $W$ in $W_0$ is positive definite. This implies that the orthogonal complement $W_{\infty}^{(\perp)}=(\bigcap_{W\in\mathcal{C}_0}{W})^{(\perp)}=\overline{\bigcup_{W\in\mathcal{C}_0}{W^{(\perp)}}}$ is also positive definite, which means that $W_{\infty}$ is nonzero and has index $k_0$. Therefore, $W_{\infty}\in\mathcal{W}$. \\

    Finally, we claim that $W_{\infty}$ is a topologically irreducible $\mathbb{C}$-subrepresentation of $G$. We have shown that any nonzero closed $G$-invariant $\mathbb{C}$-subspace of $H_\mathbb{C}$ must be nondegenerate, so any nonzero closed $G$-invariant $\mathbb{C}$-subspace $W'$ of $W_{\infty}$ would be nondegenerate, and either $W'$ or its orthogonal complement in $W_{\infty}$ would be indefinite. \\
\end{proof}

Now, in the setting of Proposition \ref{subirredC}, we can say that the irreducible $\mathbb{C}$-subrepresentation $W_{\infty}$ of $H_\mathbb{C}$ is admissible due to Theorem \ref{admissC}. This means that for every open compact $K<G$, the subspace $(W_{\infty})^K$ is finite-dimensional. We also know that $W_{\infty}$, viewed as a real representation of $G$, has a $G$-equivariant injective dense embedding into $H$ (we consider the projections to each summand of $H_\mathbb{C}=H\oplus H$, whose kernel is trivial or full and whose image is trivial or dense). Therefore, we may consider $W_{\infty}$ as a dense real $G$-subspace of $H$. Since the projection $H\rightarrow H^K$, defined as the integration over $K$, is continuous, the subspace $(W_{\infty})^K$ is dense in $H^K$. As $(W_{\infty})^K$ is finite-dimensional, this means that $(W_{\infty})^K=H^K$ and $H^K$ is finite-dimensional, which shows admissibility in the case $\mathbb{K}=\mathbb{R}$. \\

In the case $\mathbb{K}=\mathbb{H}$, we start from a continuous topologically irreducible $\mathbb{H}$-representation $(\pi,H)$ of $G$, and we can define $H_\mathbb{C}$ to be $H$ with the scalars restricted to $\mathbb{C}$; we have that the quaternionification $H_\mathbb{C}\otimes_\mathbb{C}\mathbb{H}$ is isomorphic to $H\oplus H$, and we can repeat the same proof as in the case $\mathbb{K}=\mathbb{R}$, replacing complexifications with restrictions of scalars from $\mathbb{H}$ to $\mathbb{C}$, and restrictions of scalars from $\mathbb{C}$ to $\mathbb{R}$ with quaternionifications. For further information about the interplay between real, complex, and quaternionic representations, a useful resource is (\cite{brodieck}, Chapter II.6).

\printbibliography

Federico Viola, EPFL, Lausanne, CH-1015, Switzerland.

\textit{E-mail address: federico.viola@epfl.ch}

\end{document}